\documentclass[11pt]{article}

\usepackage{amssymb,amsmath,amsfonts}
\usepackage{graphicx,color,enumitem}
\usepackage{amsthm} % For proof environment
\usepackage{bm}
\usepackage[round]{natbib}

\usepackage{geometry}

\usepackage[color]{changebar}
\cbcolor{blue}

\usepackage[colorinlistoftodos, textwidth=4cm, shadow]{todonotes}
\RequirePackage[colorlinks,citecolor=blue,urlcolor=blue]{hyperref}

\usepackage{tikz}
\tikzstyle{vertex}=[circle, draw, inner sep=2pt, fill=white]%, minimum size=6pt] 

\newcommand{\E}{{\mathbb E}}

\renewcommand{\P}{{\mathbb P}}
\newcommand{\Q}{{\mathbb Q}}
\newcommand{\C}{{\mathbb C}}
\newcommand{\R}{{\mathbb R}}
\renewcommand{\S}{{\mathbb S}}
\newcommand{\N}{{\mathbb N}}

\newcommand{\Ecal}{{\mathcal E}}
\newcommand{\Fcal}{{\mathcal F}}

\newcommand{\Hcal}{{\mathcal H}}

\newcommand{\Tcal}{{\mathcal T}}

\newcommand{\Xcal}{{\mathcal X}}

\newcommand{\Mid}{{\ \Big|\ }}

\newcommand{\lc}{[\![}
\newcommand{\rc}{]\!]}

\newcommand{\fdot}{{\,\cdot\,}}

\newcommand{\id}{{I}}

\newcommand{\iii}{{\vert\kern-0.25ex\vert\kern-0.25ex\vert}}

\newcommand{\vertiii}[1]{{\left\vert\kern-0.25ex\left\vert\kern-0.25ex\left\vert #1 
    \right\vert\kern-0.25ex\right\vert\kern-0.25ex\right\vert}}

\DeclareMathOperator{\tr}{tr}

\newtheorem{theorem}{Theorem}

\newtheorem{corollary}[theorem]{Corollary}

\newtheorem{definition}[theorem]{Definition}

\newtheorem{lemma}[theorem]{Lemma}

\newtheorem{remark}[theorem]{Remark}

\theoremstyle{definition}
\newtheorem{example}[theorem]{Example}

\numberwithin{equation}{section}
\numberwithin{theorem}{section}

\definecolor{darkgreen}{rgb}{0,0.7,0}

\DeclareMathOperator{\diag}{diag}

\begin{document}

\title{Affine Volterra processes\footnote{Acknowledgments: The authors wish to thank Bruno Bouchard, Omar El Euch, Camille Illand, and Mathieu Rosenbaum for useful comments and fruitful discussions. The authors also thank the anonymous referees for their careful reading of the manuscript and suggestions. Martin Larsson gratefully acknowledges financial support by the Swiss National Science Foundation (SNF) under grant 205121\textunderscore163425. The research of Sergio Pulido benefited from the support of the Chair Markets in Transition (F\'ed\'eration Bancaire Fran\c caise) and the project ANR 11-LABX-0019.}}
\author{Eduardo Abi Jaber\thanks{Universit\'e Paris-Dauphine, PSL Research University, CNRS, UMR [7534], CEREMADE, 75016 Paris, France and AXA Investment Managers,  Multi Asset Client Solutions, Quantitative Research,
			6 place de la Pyramide, 92908 Paris - La D\'efense, France, abijaber@ceremade.dauphine.fr.}
\and Martin Larsson\thanks{Department of Mathematical Sciences, Carnegie Mellon University, Pittsburgh, USA, martinl@andrew.cmu.edu.}
\and Sergio Pulido\thanks{Laboratoire de Math\'ematiques et Mod\'elisation d'\'Evry (LaMME), Universit\'e d'\'Evry-Val-d'Essonne, ENSIIE, Universit\'e Paris-Saclay, UMR CNRS 8071, IBGBI 23 Boulevard de France, 91037 \'Evry Cedex, France, sergio.pulidonino@ensiie.fr.}}

\maketitle

\begin{abstract}
We introduce affine Volterra processes, defined as solutions of certain stochastic convolution equations with affine coefficients. Classical affine diffusions constitute a special case, but affine Volterra processes are neither semimartingales, nor Markov processes in general. We provide explicit exponential-affine representations of the Fourier--Laplace functional in terms of the solution of an associated system of deterministic integral equations of convolution type, extending well-known formulas for classical affine diffusions. For specific state spaces, we prove existence, uniqueness, and invariance properties of solutions of the corresponding stochastic convolution equations. Our arguments avoid infinite-dimensional stochastic analysis as well as stochastic integration with respect to non-semimartingales, relying instead on tools from the theory of finite-dimensional deterministic convolution equations. Our findings generalize and clarify recent results in the literature on rough volatility models in finance.
\\[2ex] 
\noindent{\textbf {Keywords:} stochastic Volterra equations, Riccati--Volterra equations, affine processes, rough volatility.}
\\[2ex]
\noindent{\textbf {MSC2010 classifications:} 60J20 (primary), 60G22, 45D05, 91G20 (secondary).}
\end{abstract}

\section{Introduction}

We study a class of $d$-dimensional stochastic convolution equations of the form
\begin{equation} \label{SVE}
X_t = X_0 + \int_0^t K(t-s) b(X_s) ds + \int_0^t K(t-s) \sigma(X_s) dW_s,
\end{equation}
where $W$ is a multi-dimensional Brownian motion, and the convolution kernel $K$ and coefficients $b$ and $\sigma$ satisfy regularity and integrability conditions that are discussed in detail after this introduction. We refer to equations of the form \eqref{SVE} as {\em stochastic Volterra equations} (of convolution type), and their solutions are always understood to be adapted processes defined on some stochastic basis $(\Omega,\Fcal,(\Fcal_t)_{t\ge0},\P)$ satisfying the usual conditions. Stochastic Volterra equations have been studied by numerous authors; see e.g.~\cite{BM:80:1,BM:80:2,P:85,PP:90,cou_dec_01,Z:10,MS:15} among many others. In Theorem~\ref{T:EXISTENCE} and Theorem~\ref{T:existence orthant} we provide new existence results for \eqref{SVE} under weak conditions on the kernel and coefficients.

We are chiefly interested in the situation where  $a(x)=\sigma(x)\sigma(x)^\top$ and $b(x)$ are affine of the form
\begin{equation} \label{a and b intro}
\begin{aligned}
a(x) &= A^0 + x_1A^1 + \cdots + x_d A^d \\
b(x) &= b^0 + x_1b^1 + \cdots + x_d b^d ,
\end{aligned}
\end{equation}
for some $d$-dimensional symmetric matrices $A^i$ and vectors $b^i$. In this case we refer to solutions of \eqref{SVE} as {\em affine Volterra processes}. {\em Affine diffusions}, as studied in \cite{DFS:03}, are particular examples of affine Volterra processes of the form \eqref{SVE} where the convolution kernel $K\equiv\id$ is constant and equal to the $d$-dimensional identity matrix. In this paper we do not consider processes with jumps.

Stochastic models using classical affine diffusions are tractable because their Fourier--Laplace transform has a simple form. It can be written as an exponential-affine function of the initial state, in terms of the solution of a system of ordinary differential equations, known as the Riccati equations, determined by the affine maps \eqref{a and b intro}. More precisely, let $X$ be an affine diffusion of the form \eqref{SVE} with $K\equiv\id$. Then, given a $d$-dimensional row vector $u$ and under suitable integrability conditions, we have
\begin{equation}\label{chfaffine intro1}
		\E\left[ \exp\left(u X_T\right)  \Mid \Fcal_t \right] =\exp\left(\phi(T-t)+\psi(T-t) X_t\right),
\end{equation}
where the real-valued function $\phi$ and row-vector-valued function $\psi$ satisfy the Riccati equations
\begin{displaymath}
\begin{split}
\phi(t)&=\int_0^t\left(\psi(s)b_0+\frac12\psi(s)A_0\psi(s)^\top\right)\,ds \\
\psi(t) &= u + \int_0^t \left(\psi(s) B + \frac12 A(\psi(s)) \right) \,ds,
\end{split}
\end{displaymath}
with $A(u) = (u A^1u^\top, \ldots, u A^d u^\top)$ and  $B = (b^1 \ \cdots \ b^d)$. Alternatively, using the variation of constants formula on $X$ and $\psi$, one can write the Fourier--Laplace transform as
\begin{equation}\label{chfaffine intro2}
\E\left[ \exp\left(u X_T\right)  \Mid \Fcal_t \right]  =\exp\left( \E[u X_T \mid\Fcal_t] + \frac12 \int_t^T \psi(T-s) a(\E[X_s\mid\Fcal_t])\psi(T-s)^\top ds\right).
\end{equation}

For more general kernels $K$, affine Volterra processes are typically neither semimartingales, nor Markov processes. Therefore one cannot expect a formula like \eqref{chfaffine intro1} to hold in general. However, we show in Theorem \ref{T:cf} below that, remarkably, \eqref{chfaffine intro2} does continue to hold, where now the function $\psi$ solves the {\em Riccati--Volterra equation} 
\begin{equation} \label{RicVol intro}
\psi(t) = uK(t) + \int_0^t\left( \psi(s) B + \frac12 A(\psi(s)) \right) K(t-s)\,ds.
\end{equation}
Furthermore, it is possible to express \eqref{chfaffine intro2} in a form that is exponential-affine in the past trajectory $\{X_s,\,s\le t\}$. This is done in Theorem~\ref{T:aff_past}.

For the state spaces $\R^d$, $\R^d_+$, and $\R\times\R_+$, corresponding to the {\em Volterra Ornstein--Uhlenbeck}, {\em Volterra square-root}, and {\em Volterra Heston} models, we establish existence and uniqueness of global solutions of both the stochastic equation \eqref{SVE} and the associated Riccati--Volterra equation \eqref{RicVol intro}, under general parameter restrictions. For the state spaces $\R^d_+$ and $\R\times\R_+$, which are treated in Theorem~\ref{T:VolSqrt} and Theorem~\ref{T:VolterraHeston}, this involves rather delicate invariance properties for these equations. While standard martingale and stochastic calculus arguments play an important role in several places, the key tools that allow us to handle the lack of Markov and semimartingale structure are the {\em resolvents of first and second kind} associated with the convolution kernel $K$. Let us emphasize in particular that no stochastic integration with respect to non-semimartingales is needed. Furthermore, by performing the analysis on the level of finite-dimensional integral equations, we avoid the infinite-dimensional analysis used, for instance, by~\cite{MS:15}. We also circumvent the need to study scaling limits of Hawkes processes as in \cite{EER:06,EluchFukasawaRosenbaum2016,EER:07}.

Our motivation for considering affine Volterra processes comes from applications in financial modeling. Classical affine processes arguably constitute the most popular framework for building tractable multi-factor models in finance. They have been used to model a vast range of risk factors such as credit and liquidity factors, inflation and other macro-economic factors, equity factors, and factors driving the evolution of interest rates; see~\cite{DFS:03} and the references therein. In particular, affine stochastic volatility models, such as the \cite{heston1993closed} model, are very popular.

However, a growing body of empirical research indicates that volatility fluctuates more rapidly than Brownian motion, which is inconsistent with standard semimartingale affine models. Fractional volatility models such as those by \cite{com_cou_ren_12,GJR:17,volatilityrough2014,Bayeretal2016,EER:06,BLP:16} have emerged as compelling alternatives, although tractability can be a challenge for these non-Markovian, non-semimartingales models. Nonetheless, \cite{GJR:17} and \cite{EER:06,EER:07} show that there exist fractional adaptations of the Heston model where the Fourier--Laplace transform can be found explicitly, modulo the solution of a specific {\em fractional Riccati equation}. These models are of the affine Volterra type \eqref{SVE} involving singular kernels proportional to $t^{\alpha-1}$. Our framework subsumes and extends these examples.

The paper is structured as follows. Section~\ref{S:conv} covers preliminaries on convolutions and their resolvents, and in particular develops the necessary stochastic calculus. Section~\ref{S:SVE} gives existence theorems for stochastic Volterra equations on $\R^d$ and $\R^d_+$. Section~\ref{S:aff Vol} introduces affine Volterra processes on general state spaces and develops the exponential-affine transform formula. Sections~\ref{S:VOU} through~\ref{S:VH} contain detailed discussions for the state spaces $\R^d$, $\R^d_+$, and $\R\times\R_+$, which correspond to the Volterra Ornstein--Uhlenbeck, Volterra square-root, and Volterra Heston models, respectively. Additional proofs and supporting results are presented in the appendices. Our basic reference for the deterministic theory of Volterra equations is the excellent book by \cite{GLS:90}.

\subsection*{Notation}
Throughout the paper we view elements of $\R^m$ and $\C^m=\R^m+{\rm i}\R^m$ as column vectors, while elements of the dual spaces $(\R^m)^*$ and $(\C^m)^*$ are viewed as row vectors. For any matrix $A$ with complex entries, $A^\top$ denotes the (ordinary, not conjugate) transpose of $A$. The identity matrix is written $\id$. The symbol $|\fdot|$ is used to denote the Euclidean norm on $\C^m$ and $(\C^m)^*$, as well as the operator norm on $\R^{m\times n}$. We write $\S^m$ for the symmetric $m\times m$ matrices. The shift operator $\Delta_h$ with $h\ge0$, maps any function $f$ on $\R_+$ to the function $\Delta_h f$ given by
\[
\Delta_h f(t) = f(t+h).
\]
If the function $f$ on $\R_+$ is right-continuous and of locally bounded variation, the measure induced by its distributional derivative is denoted $df$, so that $f(t) = f(0) + \int_{[0,t]} df(s)$ for all $t\ge0$. By convention, $df$ does not charge $\{0\}$.

\section{Stochastic calculus of convolutions and resolvents} \label{S:conv}

For a measurable function $K$ on $\R_+$ and a measure $L$ on $\R_+$ of locally bounded variation, the convolutions $K*L$ and $L*K$ are defined by
\begin{equation} \label{K*L}
(K*L)(t) = \int_{[0,t]} K(t-s)L(ds), \qquad (L*K)(t) = \int_{[0,t]} L(ds)K(t-s)
\end{equation}
for $t>0$ whenever these expressions are well-defined, and extended to $t=0$ by right-continuity when possible. We allow $K$ and $L$ to be matrix-valued, in which case $K*L$ and $L*K$ may not both be defined (e.g.~due to incompatible matrix dimensions), or differ from each other even if they are defined (e.g.~if $K$ and $L$ take values among non-commuting square matrices). If $F$ is a function on $\R_+$, we write $K*F=K*(Fdt)$, that is,
\begin{equation} \label{K*F}
(K*F)(t) = \int_0^t K(t-s) F(s) ds.
\end{equation}
Further details can be found in \cite{GLS:90}, see in particular Definitions~2.2.1 and~3.2.1, as well as Theorems~2.2.2 and~3.6.1 for a number of properties of convolutions. In particular, if $K\in L^1_{\rm loc}(\R_+)$ and $F$ is continuous, then $K*F$ is again continuous.

Fix $d\in\N$ and let $M$ be a $d$-dimensional continuous local martingale. If $K$ is $\R^{m\times d}$-valued for some $m\in\N$, the convolution
\begin{equation} \label{K*dM_t}
(K*dM)_t = \int_0^t K(t-s)dM_s
\end{equation}
is well-defined as an It\^o integral for any $t\ge0$ that satisfies
\[
\int_0^t |K(t-s)|^2 d\tr\langle M\rangle_s<\infty.
\] 
In particular, if $K\in L^2_{\rm loc}(\R_+)$ and $\langle M\rangle_s=\int_0^s a_u du$ for some locally bounded process $a$, then \eqref{K*dM_t} is well-defined for every $t\ge0$. We always choose a version that is jointly measurable in $(t,\omega)$. Just like \eqref{K*L}--\eqref{K*F}, the convolution \eqref{K*dM_t} is associative, as the following result shows.

\begin{lemma}\label{L:assoc_dM}
Let $K\in L^2_{\rm loc}(\R_+,\R^{m\times d})$ and let $L$ be an $\R^{n\times m}$-valued measure on $\R_+$ of locally bounded variation. Let $M$ be a $d$-dimensional continuous local martingale with $\langle M\rangle_t=\int_0^t a_s ds$, $t\ge0$, for some locally bounded adapted process $a$. Then
\begin{equation} \label{L:assoc_dM_eq}
(L*(K*dM))_t = ((L*K)*dM)_t
\end{equation}
for every $t\ge0$. In particular, taking $F\in L^1_{\rm loc}(\R_+)$ we may apply \eqref{L:assoc_dM_eq} with $L(dt)=Fdt$ to obtain $(F*(K*dM))_t = ((F*K)*dM)_t$.
\end{lemma}

\begin{proof}
By linearity it suffices to take $d=m=n=1$ and $L$ a locally finite positive measure. In this case,
\[
(L*(K*dM))_t = \int_0^t \left( \int_0^t \bm1_{\{u<t-s\}}K(t-s-u)dM_u \right) L(ds).
\]
Since
\[
\int_0^t \left( \int_0^t  \bm1_{\{u<t-s\}}K(t-s-u)^2 d\langle M\rangle_u \right)^{1/2} L(ds) \le \max_{0\le s\le t}|a_s|^{1/2} \| K\|_{L^2(0,t)} L([0,t]),
\]
which is finite almost surely, the stochastic Fubini theorem, see \citet[Theorem~2.2]{V:12}, yields
\[
(L*(K*dM))_t = \int_0^t \left( \int_0^t \bm1_{\{u<t-s\}}K(t-s-u)L(ds) \right) dM_u = ((L*K)*dM)_t,
\]
as required.
\end{proof}

Under additional assumptions on the kernel $K$ one can find a version of the convolution~\eqref{K*dM_t} that is continuous in~$t$. We will use the following condition:
\begin{equation} \label{K_gamma}
\begin{minipage}[c][4em][c]{.78\textwidth}
\begin{center}
$K\in L^2_{\rm loc}(\R_+,\R)$ and there is $\gamma\in(0,2]$ such that $\int_0^h K(t)^2dt = O(h^\gamma)$
and $\int_0^T (K(t+h)-K(t))^2 dt = O(h^\gamma)$ for every $T<\infty$.
\end{center}
\end{minipage}
\end{equation}

\begin{remark}
Other conditions than \eqref{K_gamma} have appeared in the literature. \cite{D:02} considers $dM=\sigma dW$ defined on the Wiener space with coordinate process $W$, and requires $F\mapsto K*F$ to be continuous from certain $L^p$ spaces to appropriate Besov spaces. \cite{MN:11} assume $K$ to be a function of smooth variation and $M$ to be a semimartingale. See also \cite[Theorem 1.3]{W:08}.
\end{remark}

\begin{example} \label{ex:kernels}
Let us list some examples of kernels that satisfy~\eqref{K_gamma}:
\begin{enumerate}
\item Locally Lipschitz kernels $K$ clearly satisfy~\eqref{K_gamma} with $\gamma=1$.
\item The fractional kernel $K(t)=t^{\alpha-1}$ with $\alpha\in(\frac12,1)$ satisfies~\eqref{K_gamma} with $\gamma=2\alpha-1$. Indeed, it is locally square integrable, and we have $\int_0^h K(t)^2dt=h^{2\alpha-1}/(2\alpha-1)$ as well as
\[
\int_0^T (K(t+h)-K(t))^2 dt \le h^{2\alpha-1} \int_0^\infty \left( (t+1)^{\alpha-1} - t^{\alpha-1}\right)^2 dt,
\]
where the constant on the right-hand side is bounded by $\frac{1}{2\alpha-1} + \frac{1}{3-2\alpha}$. Note that the case $\alpha\ge1$ falls in the locally Lipschitz category mentioned previously.
\item\label{ex:kernels:iii} If $K_1$ and $K_2$ satisfy~\eqref{K_gamma}, then so does $K_1+K_2$.
\item If $K_1$ satisfies~\eqref{K_gamma} and $K_2$ is locally Lipschitz, then $K=K_1K_2$ satisfies~\eqref{K_gamma} with the same $\gamma$. Indeed, letting $\|K_2^2\|_{\infty,T}$ denote the maximum of $K_2^2$ over $[0,T]$ and ${\rm Lip}_T(K_2)$ the best Lipschitz constant on $[0,T]$, we have
\[
\int_0^h K(t)^2dt \le \|K_2^2\|_{\infty,h} \int_0^h K_1(t)^2dt=O(h^\gamma)
\]
and
\begin{align*}
\int_0^T (K(t+h)-K(t))^2 dt &\le 2\|K_2^2\|_{\infty,T+h} \int_0^T(K_1(t+h)-K_1(t))^2 dt \\
&\quad + 2 \|K_1\|^2_{L^2(0,T)}{\rm Lip}_{T+h}(K_2)^2 h^2 \\
&=O(h^{\gamma}).
\end{align*}
\item\label{ex:kernels:v} If $K$ satisfies~\eqref{K_gamma} and $f\in L^2_{\rm loc}(\R_+)$, then $f*K$ satisfies~\eqref{K_gamma} with the same $\gamma$. Indeed, Young's inequality gives 
\[
\int_0^h (f*K)(t)^2dt\le\|f\|_{L^1(0,h)}^2\|K\|_{L^2(0,h)}^2=O(h^\gamma)
\] 
and, using also the Cauchy--Schwarz inequality,
\begin{align*}
\int_0^T ((f*K)(t+h)-(f*K)(t))^2dt &\le 2T\|f\|_{L^2(0,T+h)}^2\|K\|_{L^2(0,h)}^2 \\
&\quad + 2 \|f\|_{L^1(0,T)}^2\|\Delta_h K-K\|_{L^2(0,T)}^2 \\
&=O(h^\gamma).
\end{align*}
\item\label{ex:kernels:vi} If $K$ satisfies~\eqref{K_gamma} and is locally bounded on $(0,\infty)$, then $\Delta_\eta K$ satisfies~\eqref{K_gamma} for any $\eta>0$. Indeed, local boundedness gives $\|\Delta_\eta K\|_{L^2(0,h)}^2 = O(h)$ and it is immediate that
\[
\int_0^T (\Delta_\eta K(t+h)-\Delta_\eta K(t))^2dt\le \int_0^{T+\eta} (K(t+h)-K(t))^2dt = O(h^\gamma).
\]
\item By combining the above examples we find that, for instance, exponentially damped and possibly singular kernels like the Gamma kernel $K(t)=t^{\alpha-1}{\rm e}^{-\beta t}$ for $\alpha>\frac12$ and $\beta\ge0$ satisfy~\eqref{K_gamma}.
\end{enumerate}
\end{example}

\begin{lemma} \label{L:Holder_bound}
Assume $K$ satisfies \eqref{K_gamma} and consider a process $X=K*(bdt + dM)$, where $b$ is an adapted process and $M$ is a continuous local martingale with $\langle M\rangle_t=\int_0^t a_sds$ for some adapted process $a$. Let $T\ge0$ and $p>\max\{2,2/\gamma\}$ be such that $\sup_{t\le T} \E[ |a_t|^{p/2} + |b_t|^p ]$ is finite. Then $X$ admits a version which is H\"older continuous on $[0,T]$ of any order $\alpha<\gamma/2-1/p$. Denoting this version again by $X$, one has
\begin{equation} \label{eq:L:Holder_bound}
\E\left[ \left( \sup_{0\le s<t\le T} \frac{|X_t-X_s|}{|t-s|^\alpha} \right)^p \right] \le c\, \sup_{t\le T} \E[ |a_t|^{p/2} + |b_t|^p ]
\end{equation}
for all $\alpha\in[0,\gamma/2-1/p)$, where $c$ is a constant that only depends on $p$, $K$, and $T$. As a consequence, if $a$ and $b$ are locally bounded, then $X$ admits a version which is H\"older continuous of any order $\alpha<\gamma/2$.
\end{lemma}

\begin{proof}
For any $p\ge2$ and any $s<t\le T<\infty$ we have
\begin{align*}
|X_t - X_s|^p &\le 4^{p-1} \left| \int_s^t K(t-u)b_u du \right|^p \\
&\quad + 4^{p-1} \left| \int_0^s \left( K(t-u) - K(s-u) \right) b_u du \right|^p \\
&\quad + 4^{p-1} \left| \int_s^t K(t-u)dM_u \right|^p \\
&\quad + 4^{p-1} \left| \int_0^s \left( K(t-u) - K(s-u) \right) dM_u \right|^p \\
&= 4^{p-1}\left(  {\bf I} + {\bf II} + {\bf III} + {\bf IV} \right).
\end{align*}
Jensen's inequality applied twice yields 
\[
 {\bf I}\le (t-s)^{p/2}\left( \int_s^{t} K(t-u)^2 du\right)^{p/2-1}\int_s^t  |b_u|^p\,K(t-u)^2 du.
\]
Taking expectations and changing variables we obtain
\begin{equation} \label{bound_I_new}
\E[\,{\bf I}\,] \le (t-s)^{p/2} \left( \int_0^{t-s} K(u)^2 du\right)^{p/2} \sup_{u\le T} \E[  |b_u|^p ].
\end{equation}
In a similar manner,
\begin{equation} \label{bound_II_new}
\E[\,{\bf II}\,] \le T^{p/2}\left( \int_0^s (K(u+t-s)-K(u))^2 du\right)^{p/2} \sup_{u\le T} \E[  |b_u|^p ].
\end{equation}
Analogous calculations relying also on the BDG inequalities applied to the continuous local martingale $\{\int_0^r K(t-u)dM_u\colon r\in[0,t]\}$ yield
\begin{equation} \label{bound_III_new}
\begin{aligned}
\E\left[ \,{\bf III}\, \right]
&\le C_p\,\E\left[ \left( \int_s^t K(t-u)^2 \,a_u\, du \right)^{p/2} \right] \\
&\le C_p\,\left( \int_0^{t-s} K(u)^2 du\right)^{p/2}  \sup_{u\le T} \E[  |a_u|^{p/2} ]
\end{aligned}
\end{equation}
and
\begin{equation} \label{bound_IV_new}
\E\left[ \,{\bf IV}\, \right] \le C_p\, \left( \int_0^s (K(u+t-s)-K(u))^2 du\right)^{p/2} \sup_{u\le T} \E[  |a_u|^{p/2} ].
\end{equation}
Combining \eqref{bound_I_new}--\eqref{bound_IV_new} with \eqref{K_gamma} leads to
\[
\E\left[|X_t - X_s|^p\right] \le c'\,\sup_{u\le T} \E[  |a_u|^{p/2} + |b_u|^p ]\,(t-s)^{\gamma p/2},
\]
where $c'$ is a constant that only depends on $p$, $K$, and $T$, but not on $s$ or $t$. Existence of a continuous version as well as the bound \eqref{eq:L:Holder_bound} now follow from the Kolmogorov continuity theorem; see \citet[Theorem~I.2.1]{RY:99}.

Finally, if $a$ and $b$ are locally bounded, consider stopping times $\tau_n\to\infty$ such that $a$ and $b$ are bounded on $\lc0,\tau_n\rc$. The process $X^n=K*(b\bm1_{\lc0,\tau_n\rc}dt+a\bm1_{\lc0,\tau_n\rc}dW)$ then has a H\"older continuous version of any order $\alpha<\gamma/2$ by the first part of the lemma, and one has $X_t=X^n_t$ almost surely on $\{t\le\tau_n\}$ for each $t$.
\end{proof}

Consider a kernel $K\in L^1_{\rm loc}(\R_+,\R^{d\times d})$. The {\em resolvent}, or {\em resolvent of the second kind}, corresponding to $K$ is the kernel $R\in L^1_{\rm loc}(\R_+;\R^{d\times d})$ such that
\begin{equation} \label{resolvent}
K*R = R*K = K - R.\footnote{Rather than \eqref{resolvent}, it is common to require $K*R=R*K=R-K$ in the definition of resolvent. We use \eqref{resolvent} to remain consistent with \cite{GLS:90}.}
\end{equation}
The resolvent always exists and is unique, and a number of properties such as (local) square integrability and continuity of the original kernel $K$ are inherited by its resolvent; see \citet[Theorems~2.3.1 and~2.3.5]{GLS:90}. Using the resolvent $R$ one can derive a variation of constants formula as shown in the following lemma.

\begin{lemma} \label{L:RX}
	Let $X$ be a continuous process, $F\colon\R_+ \to  \R^m$  a continuous function, $B \in \R^{d \times d}$ and  $Z=\int b\, dt + \int \sigma\, dW$ a continuous semimartingale with $b$ and $\sigma$ continuous and adapted. Then 
	\begin{equation*} 
	X = F + (KB)*X + K*dZ \qquad\Longleftrightarrow\qquad X= F - R_B*F + E_B*dZ,
	\end{equation*}
	where $R_B$ is the resolvent of $-KB$ and $E_B=K - R_B*K$.
\end{lemma}

\begin{proof}
	Assume that $X=F+(KB)*X + K*dZ$. Convolving this with $R_B$ and using Lemma~\ref{L:assoc_dM} yields
	\begin{align*}
	X-R_B*X &= \big(F- R_B*F)  + \big( KB - R_B*(KB)\big)*X + E_B*dZ.
	\end{align*}
	The resolvent equation \eqref{resolvent} states that $KB - R_B*(KB)=-R_B$, so that
     \begin{align}\label{L:tempEB}
     	X = F - R_B*F  + E_B*dZ.
     \end{align}
	Conversely, assume that \eqref{L:tempEB} holds.  It follows from the resolvent equation \eqref{resolvent} that  $KB - (KB)*R_B=-R_B$  and 
	$$(K B)*E_B= (KB)*(K-R_B*K) = -R_B*K.$$
	Hence, convolving both sides of \eqref{L:tempEB} with $KB$ and using Lemma \ref{L:assoc_dM} yields 
	\begin{align*}
	X- (KB)*X &= F + \left(- R_B - KB + (KB)*R_B \right)*F \\
	&\quad+ (E_B- (KB)*E_B )*dZ \\
	&= F + (E_B +R_B*K)*dZ\\
	&=  F   + K*dZ,
	\end{align*} 
	which proves that  $X = F + (KB)*X + K*dZ$.	
\end{proof}

Another object related to $K$ is its {\em resolvent of the first kind}, which is an $\R^{d\times d}$-valued measure $L$ on $\R_+$ of locally bounded variation such that
\begin{equation} \label{res_L}
K*L = L*K \equiv \id,
\end{equation}
see \citet[Definition~5.5.1]{GLS:90}. We recall that $\id$ stands for the identity matrix. Some examples of resolvents of the first and second kind are presented in Table \ref{T:summary}. A resolvent of the first kind does not always exist. When it does, it has the following properties, which play a key role in several of our arguments.

\begin{lemma} \label{L:ZX}
Let $X$ be a continuous process and $Z=\int b\, dt + \int \sigma\, dW$ a continuous semimartingale with $b$, $\sigma$, and $K*dZ$ continuous and adapted. Assume that $K$ admits a resolvent of the first kind $L$. Then 
\begin{equation} \label{L:ZX:1}
X-X_0 = K*dZ \qquad\Longleftrightarrow\qquad L*(X-X_0)=Z.
\end{equation}
In this case, for any $F\in L^2_{\rm loc}(\R_+,\C^{m\times d})$ such that $F*L$ is right-continuous and of locally bounded variation, one has
\begin{equation} \label{L:ZX:2}
F*dZ = (F*L)(0)X - (F*L)X_0 + d(F*L)*X
\end{equation}
up to $dt\otimes\P$-a.e.\ equivalence. If $F*dZ$ has a right-continuous version, then with this version \eqref{L:ZX:2} holds up to indistinguishability.
\end{lemma}

\begin{proof}
Assume $X-X_0 = K*dZ$. Apply $L$ to both sides to get
\[
L*(X-X_0)=L*(K*dZ)=(L*K)*dZ=\id*dZ=Z,
\]
where the second equality follows from Lemma~\ref{L:assoc_dM}. This proves the forward implication in \eqref{L:ZX:1}. Conversely, assume $L*(X-X_0)=Z$. Then,
\begin{align*}
\id*(X-X_0) &= (K*L)*(X-X_0) \\
&= K*(L*(X-X_0)) \\
&= K*Z \\
&= K*(\id*dZ) \\
&= \id*(K*dZ),
\end{align*}
using \citet[Theorem~3.6.1(ix)]{GLS:90} for the second equality and Lemma~\ref{L:assoc_dM} for the last equality.
Since both $X-X_0$ and $K*dZ$ are continuous, they must be equal.

To prove \eqref{L:ZX:2}, observe that the assumption of right-continuity and locally bounded variation entails that
\[
F*L = (F*L)(0) + d(F*L)*\id.
\]
Convolving this with $K$, using associativity of the convolution and \eqref{res_L}, and inspecting the densities of the resulting absolutely continuous functions, we get
\[
F = (F*L)(0)K + d(F*L)*K \quad \text{a.e.}
\]
Using \eqref{L:assoc_dM_eq} and the fact that $K*dZ=X-X_0$ by assumption, it follows that
\begin{align*}
F * dZ &= (F*L)(0)K*dZ + d(F*L)*(K*dZ) \\
&= (F*L)(0)X - (F*L)X_0 + d(F*L)*X
\end{align*}
holds $dt\otimes\P$-a.e., as claimed. The final statement is clear from right-continuity of $F*L$ and $d(F*L)*X$.
\end{proof}

\begin{table}[h!]
\centering
\begin{tabular}{c c c c }
\hline\hline
& $K(t)$ & $R(t)$ & $L(dt)$ \\ 
\hline \hline \\[0.5ex]
Constant		& $c$ & $c{\rm e}^{-ct}$ & $c^{-1} \delta_0(dt)$\\ \\
Fractional		& $c\,\frac{t^{\alpha-1}}{\Gamma(\alpha)}$ & $ct^{\alpha-1}  E_{\alpha, \alpha} (-ct^{\alpha})$ & $c^{-1}\,\frac{t^{-\alpha}}{\Gamma(1-\alpha)}dt$\\ \\
Exponential	& $c{\rm e}^{-\lambda t}$ & $c{\rm e}^{-\lambda t}{\rm e}^{-ct}$ & $c^{-1}(\delta_0(dt)  + \lambda\,dt)$\\ \\
Gamma		& $c{\rm e}^{-\lambda t} \frac{t^{\alpha-1}}{\Gamma(\alpha)} $ & $ c{\rm e}^{- \lambda t}t^{\alpha-1}  E_{\alpha, \alpha} (-ct^{\alpha})$  & $c^{-1}\,\frac{1}{\Gamma(1-\alpha)}{\rm e}^{-\lambda t} \frac{d}{dt}(t^{-\alpha} \ast {\rm e}^{\lambda t})(t)dt$ \\\\
\hline
\end{tabular}
\caption{Some kernels $K$ and their resolvents $R$ and $L$ of the second and first kind. Here $E_{\alpha,\beta}(z)=\sum_{n=0}^\infty \frac{z^n}{\Gamma(\alpha n+\beta)}$ denotes the Mittag--Leffler function, and the constant $c$ may be an invertible matrix.}
\label{T:summary}
\end{table}

\section{Stochastic Volterra equations} \label{S:SVE}

Fix $d\in\N$ and consider the stochastic Volterra equation \eqref{SVE} for a given kernel $K\in L^2_{\rm loc}(\R_+,\R^{d\times d})$, initial condition $X_0\in\R^d$, and coefficients $b\colon \R^d\to\R^d$ and $\sigma\colon\R^d\to\R^{d\times m}$, where $W$ is $m$-dimensional Brownian motion. The equation \eqref{SVE} can be written more compactly as
\[
X = X_0 + K * (b(X)dt + \sigma(X)dW).
\]
We will always require the coefficients $b$ and $\sigma$ as well as solutions of~\eqref{SVE} to be continuous in order to avoid problems with the meaning of the stochastic integral term. As for stochastic (ordinary) differential equations, we say that the stochastic Volterra equation \eqref{SVE} admits a weak solution if there exists a stochastic basis $(\Omega,\mathcal F,(\mathcal F_t)_{t \geq 0},\mathbb P)$, satisfying the usual conditions and supporting a $d$-dimensional Brownian motion $W$ and an adapted continuous process $X$ such that \eqref{SVE} holds. In this case, by an abuse of terminology, we call $X$ a weak solution of~\eqref{SVE}. We call $X$ a strong solution if in addition it is adapted to the filtration generated by $W$.

The following moment bound holds for any solution of \eqref{SVE} under linear growth conditions on the coefficients.

\begin{lemma} \label{L:moment bound}
Assume $b$ and $\sigma$ are continuous and satisfy the linear growth condition
\begin{equation} \label{LG}
| b(x) | \vee | \sigma(x) |  \le c_{\rm LG} (1 + |x|), \qquad x\in\R^d,
\end{equation}
for some constant $c_{\rm LG}$. Let $X$ be a continuous solution of \eqref{SVE} with initial condition $X_0\in\R^d$. Then for any $p\ge2$ and $T<\infty$ one has
\[
\sup_{t\le T} \E[ |X_t |^p ] \le c
\]
for some constant $c$ that only depends on $|X_0|$, $K|_{[0,T]}$, $c_{\rm LG}$, $p$ and $T$.
\end{lemma}

\begin{proof}
Let $\tau_n = \inf\{t\ge0\colon |X_t| \ge n\}\wedge T$, and observe that
\begin{equation}\label{eq_L_moment_bound_1}
|X_t |^p\bm1_{\{t<\tau_n\}}\le \left| X_0 +  \int_0^t K(t-s) \Big( b(X_s\bm1_{\{s<\tau_n\}})ds + \sigma(X_s\bm1_{\{s<\tau_n\}})dW_s\Big)\right|^p.
\end{equation}
Indeed, for $t\ge\tau_n$ the left-hand side is zero while the right-hand side is nonnegative. For $t<\tau_n$, the local behavior of the stochastic integral (see e.g.\ \citet[Corollary of Theorem~II.18]{pro_04}) implies that the right-hand side is equal to
\[
\left| X_0 +  \int_0^t K(t-s) \Big( b(X_s)ds + \sigma(X_s)dW_s\Big)\right|^p.
\]
This in turn equals $|X_t |^p$ since $X$ is assumed to be a solution of \eqref{SVE}. We deduce that \eqref{eq_L_moment_bound_1} holds.

Starting from \eqref{eq_L_moment_bound_1}, we argue as in the proof of Lemma~\ref{L:Holder_bound}. The Jensen and BDG inequalities combined with the linear growth condition~\eqref{LG} yield that the expectations $f_n(t)=\E[|X_t|^p\bm1_{\{t<\tau_n\}}]$ satisfy the inequality
\[
f_n \le c' + c' |K|^2 * f_n
\]
on $[0,T]$ for some constant $c'$ that only depends on $|X_0|$, $\|K\|_{L^2(0,T)}$, $c_{\rm LG}$, $p$ and $T$. Consider now the scalar non-convolution kernel $K'(t,s)=c'|K(t-s)|^2\bm1_{s\le t}$. This is a Volterra kernel in the sense of \citet[Definition~9.2.1]{GLS:90}, and  for any interval $[u,v]\subset\R_+$, Young's inequality implies that
	\begin{align}\label{E:K'L1}
	\vertiii{K'}_{L^1(u,v)} \le c'\| K \|_{L^2(0,v-u)},
	\end{align}
	where $\vertiii{\fdot}_{L^1(u,v)}$ is defined in \citet[Definition~9.2.2]{GLS:90}.
	Thus $-K'$ is of type $L^1$ on $(0,T)$. Next, we show that $-K'$ admits a resolvent of type $L^1$ on $(0,T)$ in the sense of \citet[Definition~9.3.1]{GLS:90}. For $v-u$ sufficiently small, the right-hand side in \eqref{E:K'L1} is smaller than $1$, whence $\iii K' \iii_{L^1(u,v)}<1$. We now apply \citet[Corollary~9.3.14]{GLS:90} to obtain a resolvent of type $L^1$ on $(0,T)$ of $-K'$, which we denote by $R'$. Since $-K'$ is a convolution kernel, so is $R'$. Since also $-c' |K|^2$ is nonpositive, it follows from  \citet[Proposition~9.8.1]{GLS:90} that $R'$ is also nonpositive. The Gronwall type inequality in \citet[Lemma~9.8.2]{GLS:90} then yields $f_n(t) \le c'(1 - (R'*1)(t))\le c'(1-(R'*1)(T))$ for $t\in[0,T]$. Sending $n$ to infinity and using Fatou's lemma completes the proof.
\end{proof}

\begin{remark} \label{R:moment bound}
It is clear from the proof that the conclusion of Lemma~\ref{L:moment bound} holds also for state and time-dependent predictable coefficients $b(x,t,\omega)$ and $\sigma(x,t,\omega)$, provided they satisfy a linear growth condition uniformly in $(t,\omega)$, that is,
\[
| b(x,t,\omega) | \vee | \sigma(x,t,\omega) |  \le c_{\rm LG} (1 + |x|), \qquad x\in\R^d,\ t\in\R_+,\ \omega\in\Omega,
\]
for some constant $c_{\rm LG}$.
\end{remark}

The following existence results can be proved using techniques based on classical methods for stochastic differential equations; the proofs are given in Section~\ref{S:proof_existence}.
\begin{theorem} \label{T:sol_lip}
Assume $b$ and $\sigma$ are Lipschitz continuous and the components of $K$ satisfy~\eqref{K_gamma}. Then \eqref{SVE} admits a unique continuous strong solution $X$ for any initial condition $X_0\in\R^d$.
\end{theorem}

\begin{theorem} \label{T:EXISTENCE}
Assume that $K$ admits a resolvent of the first kind, that the components of $K$ satisfy \eqref{K_gamma}, and that $b$ and $\sigma$ are continuous and satisfy the linear growth condition \eqref{LG}. Then \eqref{SVE} admits a continuous weak solution for any initial condition $X_0\in\R^d$.
\end{theorem}

\begin{remark}
At the cost of increasing the dimension, \eqref{SVE} also covers the superficially different equation $X=X_0+K_1*(b(X)dt) + K_2*(\sigma(X)dW)$ where the drift and diffusion terms are convolved with different kernels $K_1$ and $K_2$. Indeed, if one defines
\[
\widetilde K = \begin{pmatrix} K_1 & K_2 \\ 0 & K_2 \end{pmatrix}, \qquad \widetilde b(x,y) = \begin{pmatrix} b(x) \\ 0 \end{pmatrix}, \qquad \widetilde \sigma(x,y) = \begin{pmatrix} 0 & \sigma(x) \\ 0 & 0 \end{pmatrix},
\]
and obtains a solution $Z=(X,Y)$ of the equation $Z=Z_0 + \widetilde K*(\widetilde b(Z)dt + \widetilde\sigma(Z)d\widetilde W)$ in $\R^{2d}$, where $Z_0=(X_0,0)$ and $\widetilde W=(W',W)$ is a $2d$-dimensional Brownian motion, then $X$ is a solution of the original equation of interest. If $K_1$ and $K_2$ admit resolvents of the first kind $L_1$ and $L_2$, then
\[
\widetilde L = \begin{pmatrix} L_1 & -L_1 \\ 0 & L_2 \end{pmatrix}
\]
is a resolvent of the first kind of $\widetilde K$, and Theorem~\ref{T:EXISTENCE} is applicable.
\end{remark}

Our next existence result is more delicate, as it involves an assertion about stochastic invariance of the nonnegative orthant $\R^d_+$. This forces us to impose stronger conditions on the kernel $K$ along with suitable boundary conditions on the coefficients $b$ and $\sigma$. We note that any nonnegative and non-increasing kernel that is not identically zero admits a resolvent of the first kind; see \citet[Theorem~5.5.5]{GLS:90}.

\begin{theorem} \label{T:existence orthant}
Assume that $K$ is diagonal with scalar kernels $K_i$ on the diagonal that satisfy \eqref{K_gamma} as well as
\begin{equation} \label{eq:K orthant}
\begin{minipage}[c][4em][c]{.8\textwidth}
\begin{center}
$K_i$ is nonnegative, not identically zero, non-increasing and continuous on $(0,\infty)$, and its resolvent of the first kind $L_i$ is nonnegative and non-increasing in that $s\mapsto L_i([s,s+t])$ is non-increasing for all $t\ge0$.
\end{center}
\end{minipage}
\end{equation}
Assume also that $b$ and $\sigma$ are continuous and satisfy the linear growth condition \eqref{LG} along with the boundary conditions
\[
\text{$x_i=0$ implies $b_i(x)\ge0$ and $\sigma_i(x)=0$,}
\]
where $\sigma_i(x)$ is the $i$th row of $\sigma(x)$. Then \eqref{SVE} admits an $\R^d_+$-valued continuous weak solution for any initial condition $X_0\in\R^d_+$.
\end{theorem}

\begin{example} \label{E:CM}
If $K_i$ is completely monotone on $(0,\infty)$ and not identically zero, then \eqref{eq:K orthant} holds due to \citet[Theorem~5.5.4]{GLS:90}. Recall that a function $f$ is called completely monotone on $(0,\infty)$ if it is infinitely differentiable there with $(-1)^k f^{(k)}(t) \ge0$ for all $t>0$ and $k=0,1,\ldots$. This covers, for instance, any constant positive kernel, the fractional kernel $t^{\alpha-1}$ with $\alpha\in(\frac12,1)$, and the exponentially decaying kernel ${\rm e}^{-\beta t}$ with $\beta>0$. Moreover, sums and products of completely monotone functions are completely monotone.
\end{example}

\begin{proof}[Proof of Theorem~\ref{T:existence orthant}]
Define coefficients $b^n$ and $\sigma^n$ by
\[
b^n(x) = b\left( (x-n^{-1})^+\right), \qquad \sigma^n(x) = \sigma\left( (x-n^{-1})^+\right),
\]
and let $X^n$ be the solution of \eqref{SVE} given by Theorem~\ref{T:EXISTENCE}, with $b$ and $\sigma$ replaced by $b^n$ and $\sigma^n$. Note that $b^n$ and $\sigma^n$ are continuous, satisfy \eqref{LG} with a common constant, and converge to $b(x^+)$ and $\sigma(x^+)$ locally uniformly. Lemmas~\ref{L:tight} and~\ref{L:convergence} therefore imply that, along a subsequence, $X^n$ converges weakly to a solution $X$ of the stochastic Volterra equation
\[
X_t = X_0 + \int_0^t K(t-s) b(X_s^+) ds + \int_0^t K(t-s) \sigma(X_s^+) dW_s.
\] 
It remains to prove that $X$ is $\R^d_+$-valued and hence a solution of~\eqref{SVE}. For this it suffices to prove that each $X^n$ is $\R^d_+$-valued.

Dropping the superscript $n$, we are thus left with the task of proving the theorem under the stronger condition that, for some fixed $n\in\N$,
\begin{equation} \label{eq:orth:bdry}
\text{$x_i\le n^{-1}$ implies $b_i(x)\ge0$ and $\sigma_i(x)=0$.}
\end{equation}
Define $Z=\int b(X)dt + \int \sigma(X)dW$. For any $h>0$ and $i\in\{1,\ldots,d\}$, we have the identity
\begin{equation}\label{eq_Xt+h_Yt}
X_{i,t+h} = X_{i,0} + (K_i*dZ_i)_{t+h} = X_{i,0} + (\Delta_h K_i * dZ_i)_t + Y_t, \quad t\ge0,
\end{equation}
where we define
\[
Y_t = \int_0^\infty \bm1_{(t,t+h]}(s)K_i(t+h-s) d Z_{i,s}, \quad t\ge0.
\]
Since $\Delta_hK_i$ satisfies \eqref{K_gamma} due to Example~\ref{ex:kernels}\ref{ex:kernels:vi}, Lemma~\ref{L:Holder_bound} shows that $\Delta_h K_i*dZ_i$ has a continuous version. Thus so does $Y$, and these are the versions used in \eqref{eq_Xt+h_Yt}.
%Dropping the superscript $n$, we are thus left with the task of proving the theorem under the stronger condition that, for some fixed $n\in\N$,
%\begin{equation} \label{eq:orth:bdry}
%\text{$x_i\le n^{-1}$ implies $b_i(x)\ge0$ and $\sigma_i(x)=0$.}
%\end{equation}
%Define $Z=\int b(X)dt + \int \sigma(X)dW$. For any $h>0$, we have the identity
%\begin{equation}\label{eq_Xt+h_Yt}
%X_{t+h} = X_0 + (K*dZ)_{t+h} = X_0 + (\Delta_h K * dZ)_t + Y_t, \quad t\ge0,
%\end{equation}
%where we define
%\[
%Y_t = \int_0^\infty \bm1_{(t,t+h]}(s)K(t+h-s) d Z_s, \quad t\ge0.
%\]
%Since $\Delta_hK_i$ satisfies \eqref{K_gamma} due to Example~\ref{ex:kernels}\ref{ex:kernels:vi}, Lemma~\ref{L:Holder_bound} shows that $\Delta_h K_i*dZ$ has a continuous version. Thus so does $Y$, and these are the versions used in \eqref{eq_Xt+h_Yt}.

We claim that for any stopping time $\tau$ we have, on $\{\tau<\infty\}$, the almost sure equality
\begin{equation}\label{eq_eval_tau}
Y_\tau = \int_0^\infty \bm1_{(\tau,\tau+h]}(s)K_i(\tau+h-s) d Z_{i,s}.
\end{equation}
Note that the stochastic integral on the right-hand side is well-defined, since the integrand $\bm1_{(\tau,\tau+h]}(s)K_i(\tau+h-s)$ defines a predictable and $Z_i$-integrable process. We prove \eqref{eq_eval_tau} in the case where $\tau$ is bounded by some $T\ge0$, and assuming that $b_i=0$ so that $dZ_i=\sigma_i(X)dW$. The general case then follows easily. If $\tau$ takes finitely many values, the identity \eqref{eq_eval_tau} holds due to the local behavior of the stochastic integral. Suppose now $\tau\le T$ is arbitrary. For $k\in\N$, let $\tau_k$ be the stopping time given as the minimum of $T$ and the smallest multiple of $2^{-k}$ greater than $\tau$. Then $\tau_k$ takes finitely many values and $0\le\tau_k-\tau\le2^{-k}$. For all $k\in\N$ such that $2^{-k}<h$, we have the bound
\begin{align*}
\E\bigg[ \int_0^{T+h} & |\bm1_{(\tau_k,\tau_k+h]}(s)K_i(\tau_k+h-s) - \bm1_{(\tau,\tau+h]}(s)K_i(\tau+h-s)|^2 |\sigma_i(X_s)|^2 ds \bigg] \\
&\le \left( \int_{h-2^{-k}}^hK_i(u)^2du + \int_0^{2^{-k}}K_i(u)^2du + \int_0^h (K_i(u)-K_i(u+2^{-k}))^2du\right)\\
&\quad \times \E\left[ \sup_{s\le T+h} |\sigma_i(X_s)|^2\right].
\end{align*}
The right-hand side is finite and tends to zero as $k\to\infty$ due to Lemmas~\ref{L:Holder_bound} and \ref{L:moment bound} and the dominated convergence theorem. Thus by the It\^o isometry, $\int_0^\infty \bm1_{(\tau_k,\tau_k+h]}(s)K_i(\tau_k+h-s) d Z_{i,s}$ converges in $L^2$ to the right-hand side of \eqref{eq_eval_tau} as $k\to\infty$. Since $Y_t$ is continuous in $t$, we deduce that \eqref{eq_eval_tau} holds, as claimed.
Define the stopping time
\[
\tau_i=\inf\{t\ge0\colon X_{i,t}<0\}.
\]
Applying \eqref{eq_Xt+h_Yt} and \eqref{eq_eval_tau} with this stopping time yields, for any fixed $h>0$,
\begin{equation} \label{eq:orth:0}
X_{i,\tau_i+h} = X_{i,0} + (\Delta_h K_i * dZ_i)_{\tau_i} + \int_0^\infty \bm1_{(\tau_i,\tau_i+h]}(s)K_i(\tau_i+h-s) d Z_{i,s}
\end{equation}
on $\{\tau_i<\infty\}$. We claim that
\begin{equation} \label{eq:orth:1}
\text{$(\Delta_h K_i * L_i)(t)$ is nondecreasing in $t$.}
\end{equation}
Indeed, using that $K_i*L_i\equiv 1$ we have
\begin{align*}
(\Delta_h K_i*L_i)(t) &= \int_{[0,t]} K_i(t+h-u)L_i(du) \\
&= 1 - \int_{(t,t+h]}K_i(t+h-u)L_i(du) \\
&= 1 - \int_{(0,h]} K_i(h-u)L_i(t+du),
\end{align*}
and therefore, for any $s\le t$,
\[
(\Delta_h K_i * L_i)(t)-(\Delta_h K_i * L_i)(s) = \int_{(0,h]} K_i(h-u) \left( L_i(s+du) - L_i(t+du) \right).
\]
This is nonnegative since $K_i$ is nonnegative and $L_i$ non-increasing, proving~\eqref{eq:orth:1}.  Furthermore, since $K_i$ is non-increasing and $L_i$ nonnegative we obtain 
\begin{equation} \label{eq:orth:2}
0 \le (\Delta_hK_i * L_i)(t) \le (K_i * L_i)(t) = 1.
\end{equation}
Since $\Delta_hK_i$ is  continuous and of locally bounded variation on $\R_+$, it follows that $\Delta_hK_i*L_i$ is right-continuous and of locally bounded variation. Moreover, as remarked above, $\Delta_h K_i*dZ_i$ has a continuous version. Thus \eqref{L:ZX:2} in Lemma~\ref{L:ZX}, along with \eqref{eq:orth:1}--\eqref{eq:orth:2} and the fact that $X_{i,t}\ge0$ for $t\le\tau_i$, yield
\begin{align*}
X_{i,0} + (\Delta_h K_i * dZ_i)_{\tau_i} &= \left( 1-(\Delta_h K_i * L_i)(\tau_i)\right)X_{i,0} \\
&\quad + (\Delta_h K_i * L_i)(0)X_{i,\tau_i} \\
&\quad + (d(\Delta_h K_i * L_i)*X_i)_{\tau_i} \\
&\ge 0.
\end{align*}
In view of \eqref{eq:orth:0} it follows that
\begin{equation} \label{eq:orth:3}
X_{i,\tau_i+h} \ge
\int_0^\infty \bm1_{(\tau_i,\tau_i+h]}(s)K_i(\tau_i+h-s)\left( b_i(X_s)ds + \sigma_i(X_s)dW_s \right)
\end{equation}
on $\{\tau_i<\infty\}$. 

Next, for every $\varepsilon>0$, define the event
\[
A_\varepsilon = \{\text{$b_i(X_s)\ge0$ and $\sigma_i(X_s)=0$ for all $s\in[\tau_i,\tau_i+\varepsilon)$}\}.
\]
We now argue that
\begin{equation}\label{eq_Ptau_i00000}
\P(\{\tau_i<\infty\}\cap A_\varepsilon)=0.
\end{equation}
Indeed, on $\{\tau_i<\infty\}\cap A_\varepsilon$, the local behavior of the stochastic integral and \eqref{eq:orth:3} yield $X_{i,\tau_i+h}\ge0$ for each fixed $h\in(0,\varepsilon)$. Then, almost surely, this holds simultaneously for all $h\in\Q\cap(0,\varepsilon)$, and, by continuity, simultaneously for all $h\in(0,\varepsilon)$. On the other hand, by definition of $\tau_i$, on $\{\tau_i<\infty\}$ one has $X_{i,\tau_i+h}<0$ for some $h\in(0,\varepsilon)$. We deduce \eqref{eq_Ptau_i00000}.

Finally, by \eqref{eq:orth:bdry} and continuity of $X$, we have $\P(\bigcup_{\varepsilon\in\Q\cap(0,1)}A_{\varepsilon})=1$. In view of \eqref{eq_Ptau_i00000}, it follows that $\tau_i=\infty$ almost surely. Since $i$ was arbitrary, $X$ is $\R^d_+$-valued as desired.
\end{proof}

\section{Affine Volterra processes} \label{S:aff Vol}

Fix a dimension $d\in\N$ and a kernel $K\in L^2_{\rm loc}(\R_+,\R^{d\times d})$. Let $a\colon\R^d\to\S^d$ and $b\colon\R^d\to\R^d$ be affine maps given by
\begin{equation} \label{a and b}
\begin{aligned}
a(x) &= A^0 + x_1A^1 + \cdots + x_d A^d \\
b(x) &= b^0 + x_1b^1 + \cdots + x_d b^d 
\end{aligned}
\end{equation}
for some $A^i\in\S^d$ and $b^i\in\R^d$, $i=0,\ldots,d$. To simplify notation we introduce the $d\times d$ matrix
\[
B = \begin{pmatrix} \, b^1 & \cdots & b^d \, \end{pmatrix},
\]
and for any row vector $u\in(\C^d)^*$ we define the row vector
\[
A(u) = (uA^1u^\top, \ldots, uA^du^\top).
\]
Let $E$ be a subset of $\R^d$, which will play the role of state space for the process defined below, and assume that $a(x)$ is positive semidefinite for every $x\in E$. Let $\sigma\colon\R^d\to\R^{d\times d}$ be continuous and satisfy $\sigma(x)\sigma(x)^\top=a(x)$ for every $x\in E$. For instance, one can take $\sigma(x) = \sqrt{\pi(a(x))}$, where $\pi$ denotes the orthogonal projection onto the positive semidefinite cone, and the positive semidefinite square root is understood.

\begin{definition}
An {\em affine Volterra process (with state space $E$)} is a continuous $E$-valued solution $X$ of \eqref{SVE} with $a=\sigma\sigma^\top$ and $b$ as in \eqref{a and b}. In this paper we always take $X_0$ deterministic.
\end{definition}

Setting $K\equiv\id$ we recover the usual notion of an affine diffusion with state space $E$; see e.g.~\citet{F:09}. Even in this case, existence and uniqueness is often approached by first fixing a state space $E$ of interest, and then studying conditions on $(a,b)$ under which existence and uniqueness can be proved; see e.g.~\cite{DFS:03,CFMT:11,SV:12,KL:17}. A key goal is then to obtain explicit parameterizations that can be used in applications. In later sections we carry out this analysis for affine Volterra processes with state space $\R^d$, $\R^d_+$, and $\R\times\R_+$. In the standard affine case more general results are available. \citet{SV:10} characterize existence and uniqueness of affine jump-diffusions on closed convex state spaces, while \citet{AJBI:16} provide necessary and sufficient first order geometric conditions for existence of affine diffusions on general closed state spaces. We do not pursue such generality here for affine Volterra processes.

Assuming that an affine Volterra process is given, one can however make statements about its law. In the present section we develop general results in this direction. We start with a formula for the conditional mean. This is an immediate consequence of the variation of constants formula derived in Lemma~\ref{L:RX}.

\begin{lemma} \label{L:mean}
Let $X$ be an affine Volterra process. Then for all $t\le T$,
\begin{align}
\E[X_T \mid\Fcal_t ] = \left(\id - \int_0^T R_B(s)ds\right) X_0 + \left( \int_0^T E_B(s)ds\right) b^0  + \int_0^t E_B(T-s)\sigma(X_s)dW_s, \label{E:EtXT1}
\end{align}
where $R_B$ is the resolvent of $-KB$ and $E_B=K - R_B*K$. In particular,
\[
\E[X_T ] = \left(\id - \int_0^T R_B(s)ds\right) X_0 + \left( \int_0^T E_B(s)ds\right) b^0.
\]
\end{lemma}

\begin{proof}
Since $X=X_0+(KB)*X + K*(b^0dt+\sigma(X)dW)$, Lemma \ref{L:RX} yields 
\[
X = \left(\id - R_B* \id\right) X_0 + E_B*(b^0dt+\sigma(X)dW).
\]
Consider the local martingale $M_t=\int_0^t E_B(T-s)\sigma(X_s)dW_s$, $t\in[0,T]$. Its quadratic variation satisfies
\[
\E[|\langle M\rangle_T|]\le\int_0^T |E_B(T-s)|^2\,\E[|\sigma(X_s)|^2]ds \le \|E_B\|_{L^2(0,T)}\max_{s\le T}\E[|\sigma(X_s)|^2],
\]
which is finite by Lemma~\ref{L:moment bound}. Thus $M$ is a martingale, so taking $\Fcal_t$-conditional expectations completes the proof.
\end{proof}

The first main result of this section is the following theorem, which expresses the conditional Fourier--Laplace functional of an affine Volterra process in terms of the conditional mean in Lemma~\ref{L:mean} and the solution of a quadratic Volterra integral equation, which we call a {\em Riccati--Volterra equation}.

\begin{theorem} \label{T:cf}
Let $X$ be an affine Volterra process and fix some $T<\infty$, $u\in(\C^d)^*$, and $f\in L^1([0,T],(\C^d)^*)$. Assume $\psi\in L^2([0,T],(\C^d)^*)$ solves the Riccati--Volterra equation
\begin{equation} \label{RicVol}
\psi = uK + \left( f + \psi B + \frac12 A(\psi) \right) * K.
\end{equation}
Then the process $\{Y_t,\,0\le t\le T\}$ defined by
\begin{align}
Y_t &= Y_0 + \int_0^t \psi(T-s)\sigma(X_s)dW_s - \frac12 \int_0^t \psi(T-s)a(X_s)\psi(T-s)^\top ds, \label{eq:dY} \\
Y_0 &= uX_0 + \int_0^T \left( f(s)X_0 + \psi(s)b(X_0) + \frac12 \psi(s) a(X_0) \psi(s)^\top \right) ds \label{eq:Y0} 
\end{align}
satisfies
\begin{equation} \label{Y}
Y_t =  \E\left[uX_T + (f*X)_T \mid\Fcal_t\right] + \frac12 \int_t^T \psi(T-s)a(\E[X_s\mid\Fcal_t])\psi(T-s)^\top ds
\end{equation}
for all $0\le t\le T$. The process $\{\exp(Y_t),\,0\le t\le T\}$ is a local martingale and, if it is a true martingale, one has the exponential-affine transform formula
\begin{equation} \label{eq:expaff}
\E\left[ \exp\left(uX_T + (f * X)_T\right)  \Mid \Fcal_t \right] = \exp(Y_t), \quad t\le T.
\end{equation}
\end{theorem}

Referring to \eqref{eq:expaff} as an exponential-affine transform formula is motivated by the fact that $Y_t$ depends affinely on the conditional expectations $\E[X_s\mid\Fcal_t]$. We show in Theorem~\ref{T:aff_past} below that under mild additional assumptions on $K$, $Y_t$ is actually an affine function of the past trajectory $\{X_s,\,s\le t\}$. Before proving Theorem~\ref{T:cf} we give the following lemma.

\begin{lemma} \label{L:RicVol_res}
The Riccati--Volterra equation \eqref{RicVol} is equivalent to
\begin{equation} \label{RicVol_2}
\psi = uE_B +  \left( f + \frac12 A(\psi)  \right)*E_B,
\end{equation}
where $E_B=K - R_B*K$ and $R_B$ is the resolvent of $-KB$.
\end{lemma}

\begin{proof}
Assume \eqref{RicVol_2} holds. Using the identity $E_B*(BK)=-R_B*K$ we get
\[
\psi - \psi*(BK) = u(E_B+R_B*K) + \left( f + \frac12 A(\psi)  \right)*(E_B+R_B*K),
\]
which is \eqref{RicVol}. Conversely, assume \eqref{RicVol} holds. With $\widetilde R_B$ being the resolvent of $-BK$, we obtain
\[
\psi - \psi*\widetilde R_B = u(K-K*\widetilde R_B) + \left( f + \frac12 A(\psi)  \right) * (K-K*\widetilde R_B) - \psi*\widetilde R_B.
\]
To deduce \eqref{RicVol_2} it suffices to prove $K*\widetilde R_B = R_B*K$. Equivalently, we show that for each $T<\infty$, there is some $\sigma>0$ such that
\begin{equation} \label{eq:LRicVol_1}
\text{$({\rm e}^{-\sigma t}K)*({\rm e}^{-\sigma t}\widetilde R_B) = ({\rm e}^{-\sigma t}R_B)*({\rm e}^{-\sigma t}K)$ on $[0,T]$},
\end{equation}
where ${\rm e}^{-\sigma t}$ is shorthand for the function $t\mapsto {\rm e}^{-\sigma t}$. It follows from the definitions that ${\rm e}^{-\sigma t}R_B$ is the resolvent of $-{\rm e}^{-\sigma t}KB$, and that ${\rm e}^{-\sigma t}\widetilde R_B$ is the resolvent of $-{\rm e}^{-\sigma t}BK$; see \citet[Lemma~2.3.3]{GLS:90}. Choosing $\sigma$ large enough that $\|{\rm e}^{-\sigma t}KB\|_{L^1(0,T)}<1$ we get, as in the proof of \citet[Theorem~2.3.1]{GLS:90},
\[
{\rm e}^{-\sigma t}R_B = -\sum_{k\ge1}({\rm e}^{-\sigma t}KB)^{*k} \quad\text{and}\quad {\rm e}^{-\sigma t}\widetilde R_B = -\sum_{k\ge1}({\rm e}^{-\sigma t}BK)^{*k}
\]
on $[0,T]$. Since $B$ does not depend on $t$ 
\[
({\rm e}^{-\sigma t}KB)^{*k}*({\rm e}^{-\sigma t}K)=({\rm e}^{-\sigma t}K)*({\rm e}^{-\sigma t}BK)^{*k},\quad k\ge 1.
\]
This readily implies \eqref{eq:LRicVol_1}, as required.
\end{proof}

\begin{proof}[Proof of Theorem~\ref{T:cf}]
Let $\widetilde Y_t$ be defined by the right-hand side of \eqref{Y} for $0\le t\le T$. We first prove that $\widetilde Y_0=Y_0$. A calculation using the identity $v a(x)v^\top = v A^0 v^\top + A(v)x$ and the definition \eqref{eq:Y0} of $Y_0$ yields
\begin{equation}\label{eq:Ric_class_D_new}
\begin{aligned}
\widetilde Y_0-Y_0 &= u\,\E[X_T-X_0] + (f*\E[X-X_0])(T) \\
&\quad + \left(\frac12 A(\psi)*\E[X-X_0]\right)(T) - \left( \psi*(b^0+BX_0)\right)(T),
\end{aligned}
\end{equation}
where $\E[X-X_0]$ denotes the function $t\mapsto\E[X_t-X_0]=\E[X_t-X_0\mid\Fcal_0]$. This function satisfies
\[
\E[X-X_0] = K*\left(b^0 + B\,\E[X]\right),
\]
as can be seen by taking expectations in \eqref{SVE} and applying the Fubini theorem thanks to Lemma~\ref{L:moment bound}. Consequently,
\begin{align*}
\frac12 A(\psi)*\E[X-X_0] &= \frac12 A(\psi)*K*\left(b^0 + B\,\E[X]\right) \\
&= \left(\psi - uK - (f+\psi B)*K \right)*\left(b^0 + B\,\E[X]\right) \\
&= \psi*\left(b^0 + B\,\E[X]\right) - u\,\E[X-X_0] \\
&\quad - (f+\psi B)*\E[X-X_0].
\end{align*}
Substituting this into \eqref{eq:Ric_class_D_new} yields $\widetilde Y_0-Y_0=0$, as required.

We now prove that $\widetilde Y=Y$. In the remainder of the proof, we let $C$ denote a quantity that does not depend on $t$, and may change from line to line.  Using again the identity $va(x)v^\top = v A^0 v^\top + A(v)x$ we get
\begin{align*}
\widetilde Y_t &= C + u\, \E[X_T\mid\Fcal_t] + \int_0^T \left( f + \frac12 A(\psi)  \right)\!(T-s)\, \E[X_s\mid\Fcal_t]\, ds \\
&\quad   - \frac12 \int_0^t \psi(T-s)a(X_s)\psi(T-s)^\top ds.
\end{align*}
Lemma~\ref{L:mean}, the stochastic Fubini theorem, see \citet[Theorem~2.2]{V:12}, and a change of variables yield
\begin{align*}
& \int_0^T \left( f + \frac12 A(\psi)  \right)\!(T-s)\, \E[X_s\mid\Fcal_t]\, ds \\
&\quad=C +  \int_0^T \left( f + \frac12 A(\psi)  \right)\!(T-s)\, \int_0^t \bm1_{\{r<s\}}E_B(s-r)\sigma(X_r)dW_r\, ds \\
&\quad=C +   \int_0^t  \left( \int_r^T \left( f + \frac12 A(\psi)  \right)\!(T-s)E_B(s-r)ds\right)\sigma(X_r)dW_r \\
&\quad=C +   \int_0^t  \left(  \left( f + \frac12 A(\psi)  \right)*E_B\right)\!(T-r)\sigma(X_r)dW_r,
\end{align*}
where the application of the stochastic Fubini theorem in the second equality is justified by the fact that
\begin{align*}
&\int_0^T \left( \int_0^t \left| \left( f + \frac12 A(\psi)  \right)\!(T-s)\, \bm1_{\{r<s\}}E_B(s-r) \sigma(X_r) \right|^2 dr \right)^{1/2} ds \\
&\quad\le \max_{0\le s\le T}|\sigma(X_s)|\, \| E_B \|_{L^2(0,T)}\| f + \frac12 A(\psi) \|_{L^1(0,T)} < \infty.
\end{align*}
Since $\E[X_T\mid\Fcal_t] = C + \int_0^t E_B(T-r)\sigma(X_r)dW_r$ by Lemma~\ref{L:mean}, we arrive at
\begin{align*}
\widetilde Y_t &= C + \int_0^t \left(uE_B +  \left( f + \frac12 A(\psi)  \right)*E_B\right)\!(T-r)\sigma(X_r)dW_r \\
&\quad   - \frac12 \int_0^t \psi(T-s)a(X_s)\psi(T-s)^\top ds.
\end{align*}
Evaluating this equation at $t=0$, we find that the quantity $C$ appearing on the right-hand side is equal to $\widetilde Y_0$, which we already proved is equal to $Y_0$. Due to Lemma~\ref{L:RicVol_res} and the definition \eqref{eq:dY} of $Y_t$, we then obtain that $\widetilde Y=Y$.
%\begin{align*}
%\widetilde Y_t &= \widetilde Y_0 + \int_0^t \left(uE_B +  \left( f + \frac12 A(\psi)  \right)*E_B\right)\!(T-r)\sigma(X_r)dW_r \\
%&\quad   - \frac12 \int_0^t \psi(T-s)a(X_s)\psi(T-s)^\top ds.
%\end{align*}
%Due to Lemma~\ref{L:RicVol_res} and \eqref{eq:Y0} we then obtain \eqref{eq:dY}.

The final statements are now straightforward. Indeed, \eqref{eq:dY} shows that $Y+\frac12\langle Y\rangle$ is a local martingale, so that $\exp(Y)$ is a local martingale by It\^o's formula. In the true martingale situation, the exponential-affine transform formula then follows upon observing that $Y_T=uX_T + (f * X)_T$ by \eqref{Y}.
\end{proof}

In the particular case $f\equiv0$ and $t=0$, Theorem~\ref{T:cf} yields two different expressions for the  Fourier--Laplace transform of $X$,
\begin{align}
\E[{\rm e}^{uX_T}] &= \exp\left( \E[uX_T]  + \frac12 \int_0^T \psi(T-t)a(\E[X_t])\psi(T-t)^\top dt \right) \label{expaff_classical_0} \\
&=\exp \left(\phi(T) + \chi(T) X_0 \right), \label{expaff_classical}
\end{align}
where $\phi$ and $\chi$ are defined by
\begin{align}
\phi'(t) &= \psi(t) b^0 + \frac12 \psi(t) A^0 \psi(t)^\top, && \phi(0)=0, \label{E:phichi-1} \\
\chi'(t) &= \psi(t) B + \frac12 A(\psi(t)),		&& \chi(0)=u. \label{E:phichi-2}
\end{align}
If $K$ admits a resolvent of the first kind $L$, one sees upon convolving \eqref{RicVol} by $L$ and using \eqref{res_L} that $ \chi = \psi *L$; see also Example~\ref{E:faffine} below. Note that \eqref{E:phichi-1}--\eqref{E:phichi-2} reduce to the classical Riccati equations when $K\equiv\id$, since in this case $L=\delta_0 \id$ and hence $\psi=\chi$. While the first expression~\eqref{expaff_classical_0} does exist in the literature on affine diffusions in the classical case $K\equiv\id$, see \citet[Proposition~4.2]{SV:10}, the second expression~\eqref{expaff_classical} is much more common.

In the classical case one has a conditional version of \eqref{expaff_classical}, namely
\[
\E[{\rm e}^{uX_T}\mid\Fcal_t] = \exp\left( \phi(T-t) + \psi(T-t)X_t \right).
\]
This formulation has the advantage of showing clearly that the right-hand side depends on $X_t$ in an exponential-affine manner. In the general Volterra case the lack of Markovianity precludes such a simple form, but using the resolvent of the first kind it is still possible to obtain an explicit expression that is exponential-affine in the past trajectory $\{X_s,\,s\le t\}$. Note that this property is not at all obvious either from \eqref{eq:expaff} or from the expression
\begin{equation*} 
\E[{\rm e}^{uX_T}\mid\Fcal_t] = \Ecal\Big(Y_0 + \int \psi(T-s)\sigma(X_s)dW_s \Big)_t,
\end{equation*}
which follows directly from \eqref{eq:dY}--\eqref{eq:Y0} and where $\Ecal$ denotes stochastic exponential. The second main result of this section directly leads to such an exponential-affine representation under mild additional assumptions on $K$.

\begin{theorem} \label{T:aff_past}
Assume $K$ is continuous on $(0,\infty)$, admits a resolvent of the first kind $L$, and that one has the total variation bound
\begin{equation} \label{T:aff_past:TVbound}
\sup_{h\le T} \|\Delta_hK*L\|_{{\rm TV}(0,T)} < \infty
\end{equation}
for all $T\ge0$. Then the following statements hold:
\begin{enumerate}
\item\label{T:aff_past:1} With the notation and assumptions of Lemma~\ref{L:mean}, the matrix function
\[
\Pi_h = \Delta_h E_B * L - \Delta_h(E_B*L)
\]
is right-continuous and of locally bounded variation on $[0,\infty)$ for every $h\ge0$, and the conditional expectation \eqref{E:EtXT1} is given by
\begin{equation}\label{E:EtXT2}
\E[X_T \mid\Fcal_t ] =  (\id*E_B)(h) b^0 + (\Delta_{h}E_B \ast L)(0)X_t  -\Pi_h(t) X_0 +  \left(  d\Pi_h \ast X  \right)_t
\end{equation}
with $h=T-t$.
\item\label{T:aff_past:2} With the notation and assumptions of Theorem~\ref{T:cf}, the scalar function
		\[
		\pi_h= \Delta_h \psi *L - \Delta_h(\psi*L)
		\]
	 is right-continuous and of bounded variation on $[0,T-h]$ for every $h\le T-t$, and the process $Y$ in~\eqref{Y} is given by
\begin{equation}\label{L:aff_past:2:eq}
Y_t  =  \phi(h) + (\Delta_h f*X)_t + (\Delta_h\psi *L)(0)X_t - \pi_h(t)X_0 + (d\pi_h * X)_t
\end{equation}
with $h=T-t$ and
	\[
	\phi(h)=\int_0^{h}\left(\psi(s)b_0+\frac12\psi(s)A_0\psi(s)^\top\right)ds.
	\]
\end{enumerate}
\end{theorem}

\begin{proof}
\ref{T:aff_past:1}: We wish to apply Lemma~\ref{L:ZX} with $F=\Delta_h E_B$ for any fixed $h\ge0$, so we first verify its hypotheses. 
Throughout the proof we will use the following identity for shifted convolutions
\begin{equation}\label{eq:shiftconvol}
\Delta_h(f * g)(t)=(\Delta_h f * g)(t)+(f*\Delta_t g)(h).
\end{equation}
Applying the shift operator $\Delta_h$ to the identity $E_B=K-R_B*K$ and using~\eqref{eq:shiftconvol} leads to
\[
\Delta_h E_B(t)=\Delta_h K(t)-(\Delta_h R_B*K)(t)-(R_B*\Delta_t K)(h).
\]
Convolving with $L$ and using the Fubini theorem yields 
\[
(\Delta_hE_B * L)(t) = (\Delta_hK*L)(t) - \int_h^{t+h}R_B(s)ds - \int_0^h R_B(h-s)(\Delta_s K*L)(t)ds.
\]
Owing to \eqref{T:aff_past:TVbound} we get the bound
\begin{equation}\label{E:TVEB}
\sup_{h\le T} \|\Delta_hE_B*L\|_{{\rm TV}(0,T)} < \infty,
\end{equation}
and using continuity of $K$ on $(0,\infty)$ we also get that $\Delta_hE_B*L$ is right-continuous on $\R_+$. In particular, in view of the identity
\begin{equation} \label{eq:L:aff_past:EB}
E_B*L=\id - R_B*\id,
\end{equation}
we deduce that $\Pi_h$ is right-continuous and of locally bounded variation as stated. Now, observe that $E_B=K-R_B*K$ is continuous on $(0,\infty)$, since this holds for $K$ and since $R_B$ and $K$ are both in $L^2_{\rm loc}$. Moreover, Example~\ref{ex:kernels}\ref{ex:kernels:iii} and~\ref{ex:kernels:v} imply that the components of $E_B$ satisfy \eqref{K_gamma}. As a result, Example~\ref{ex:kernels}\ref{ex:kernels:vi} shows that the components of $\Delta_hE_B$ satisfy \eqref{K_gamma} for any $h\ge0$. Fix $h=T-t$ and define
\[
{\textstyle Z=\int b(X)dt + \int \sigma(X)dW.}
\]
It follows from Lemma~\ref{L:Holder_bound} that $\Delta_hE_B*dZ$ has a continuous version. Lemma~\ref{L:ZX} with $F=\Delta_h E_B$ yields
\[
\Delta_h E_B * dZ = (\Delta_hE_B*L)(0)X - (\Delta_hE_B*L)X_0 + d(\Delta_hE_B*L)*X.
\]
Moreover, rearranging \eqref{E:EtXT1} and using \eqref{eq:L:aff_past:EB} gives
\[
\E[X_T\mid\Fcal_t] = (E_B*L)(T)X_0 + (E_B*\id)(h)b^0 + ( \Delta_h E_B * (dZ-BXdt))_t.
\]
Combining the previous two equalities and using the definition of $\Pi_h$ yields
\begin{align*}
\E[X_T\mid\Fcal_t] &= (E_B*\id)(h)b^0 + (\Delta_hE_B*L)(0)X_t - \Pi_h(t)X_0 \\
&\quad + ((d(\Delta_hE_B*L)-\Delta_hE_BBdt)*X)_t.
\end{align*}
The definition of $E_B$ and the resolvent equation \eqref{resolvent} show that $E_BB=-R_B$, which in combination with \eqref{eq:L:aff_past:EB} gives $E_BBdt=d(E_B*L)$. Thus \eqref{E:EtXT2} holds as claimed. This completes the proof of \ref{T:aff_past:1}.

\ref{T:aff_past:2}: 
Recall that Lemma~\ref{L:RicVol_res} gives $\psi=uE_B+G(\psi)*E_B$ where
\[
G(\psi)=f+\frac12A(\psi).
\]
Manipulating this equation and using the identity~\eqref{eq:shiftconvol} gives
\[
\Delta_h\psi(t)=u\Delta_h E_B(t) + (G(\psi)*\Delta_{t}E_B)(h)+(G(\Delta_h\psi)*E_B)(t).
\]
Convolving with $L$ and using Fubini yields
\begin{equation}\label{E:deltahpsiL}
\begin{aligned}
(\Delta_h\psi*L)(t)&=u(\Delta_h E_B*L)(t) + (G(\psi)*(\Delta_{\bullet}E_B*L)(t))(h) \\
&\quad+(G(\Delta_h\psi)*E_B*L)(t),
\end{aligned}
\end{equation}
where $(\Delta_{\bullet}E_B*L)(t)$ denotes the function $s\mapsto (\Delta_sE_B*L)(t)$. Similarly,
\begin{align*}
\Delta_h(\psi*L)(t)&=u\Delta_h(E_B*L)(t)+(G(\psi)*\Delta_{\bullet}(E_B*L)(t))(h)\\
&\quad +(G(\Delta_h\psi)*E_B*L)(t).
\end{align*}
Computing the difference between the previous two expressions gives
\begin{equation}\label{E:piandPi}
\pi_h(t)=u\Pi_h(t)+(G(\psi)*\Pi_{\bullet}(t))(h).
\end{equation}
In combination with \eqref{E:TVEB} and \eqref{eq:L:aff_past:EB}, as well as the properties of $\Pi_h$ that we have already proved, it follows that $\pi_h$ is right-continuous and of bounded variation as stated. Now, using Fubini we get
\begin{equation}\label{E:deltaf}
\E[(f*X)_T \mid\Fcal_t]=(\Delta_{T-t}f*X)_t+\int_0^{T-t}f(s)\E[X_{T-s}\mid\Fcal_t]\,ds.
\end{equation}
Combining~\eqref{Y}, \eqref{E:EtXT2}, and~\eqref{E:deltaf}, we obtain after some computations
\begin{equation}\label{E:TaffinepastI-V}
\begin{split}
Y_t &=(\Delta_{T-t}f*X)_t+\frac12\int_0^{T-t}\psi(s)A^0\psi(s)^\top\,ds\\
&\quad+\left(u(\id*E_B)(T-t) +\int_0^{T-t}G(\psi(s))(\id*E_B)(T-t-s)\,ds\right)b^0\\
&\quad+\left(u(\Delta_{T-t}E_B*L)(0) +\int_0^{T-t}G(\psi(s))(\Delta_{T-t-s}E_B*L)(0)\,ds\right)X_t\\
&\quad-\left(u\Pi_{T-t}(t)+\int_0^{T-t}G(\psi(s))\Pi_{T-t-s}(t)\,ds\right)X_0\\
&\quad+u(d\Pi_{T-t}*X)_t+\int_0^{T-t}G(\psi(s))(d\Pi_{T-t-s}*X)_t\,ds\\
&={\bf I} + {\bf II} + {\bf III} + {\bf IV}+{\bf V}.
\end{split}
\end{equation}
Here
\begin{equation}\label{E:IandII}
\begin{split}
{\bf I}+{\bf II}&=(\Delta_{T-t}f*X)_t+\frac12\int_0^{T-t}\psi(s)A^0\psi(s)^\top\,ds\\
&\quad +\left(\left(uE_B +G(\psi)*E_B\right)b^0*1\right)(T-t)\\
&=(\Delta_{T-t}f*X)_t+\phi(T-t).
\end{split}
\end{equation}
As a result of~\eqref{eq:L:aff_past:EB}, $E_B*L$ is continuous on $\R_+$, whence $(G(\Delta_h\psi)*E_B*L)(0)=0$. Evaluating~\eqref{E:deltahpsiL} at $t=0$ thus gives
\begin{equation}\label{E:III}
{\bf III}=(\Delta_{T-t}\psi*L)(0)X_t.
\end{equation}
As a consequence of~\eqref{E:piandPi},
\begin{equation}\label{E:IV}
{\bf IV}=-\pi_{T-t}(t)X_0.
\end{equation}
Finally, it follows from~\eqref{E:piandPi} that $d\pi_h=ud\Pi_h+\mu_h$, where $\mu_h(dt)=(G(\psi)*d\Pi_{\bullet}(dt))(h)$. Since for any bounded function $g$ on $[0,t]$ we have
\[
\int_{[0,t]}g(r)\mu_h(dr)=\int_0^h G(\psi(s))\left(\int_0^t g(r)d\Pi_{h-s}(dr)\right)\,ds,
\] 
we obtain
\begin{equation}\label{E:V}
{\bf V}=(d\pi_{T-t}*X)_t.
\end{equation}
Combining~\eqref{E:TaffinepastI-V}--\eqref{E:V} yields~\eqref{L:aff_past:2:eq} and completes the proof.
\end{proof}

\begin{remark}
Consider the classical case $K\equiv\id$. Then $L(dt)=\id\,\delta_0(dt)$, $R_B(t)=-B {\rm e}^{Bt}$, and $E_B(t)={\rm e}^{Bt}$. Thus $(\Delta_hE_B *L)(t)= {\rm e}^{B(t+h)}=\Delta_h(E_B *L)(t)$, so that \eqref{E:EtXT2} reduces to the well known expression $\E[X_T \mid\Fcal_t ] = {\rm e}^{B (T-t) } X_t + \int_0^{T-t} {\rm e}^{Bs} b^0ds$. In addition, in \eqref{L:aff_past:2:eq} the correction $\pi_h$ vanishes so that, if $f\equiv0$, the expression for $Y_t$ reduces to the classical form $\phi(T-t)+\psi(T-t)X_t$. 
\end{remark}

\begin{example}[Fractional affine processes]\label{E:faffine}
Let $K=\diag(K_1,\ldots,K_d)$, where
\[
K_i(t)=\frac{t^{\alpha_i-1}}{\Gamma(\alpha_i)}
\]
for some $\alpha_i\in(\frac12,1]$. Then $L=\diag(L_1,\ldots,L_d)$ with $L_i(dt) = \frac{t^{-\alpha_i}}{\Gamma(1-\alpha_i)}dt$ if $\alpha_i<1$, and $L_i(dt)= \delta_0(dt)$ if $\alpha_i=1$. It follows that $\chi_i = \psi_i * L_i = I^{1 - \alpha_i} \psi_i$, where $I^{1-\alpha_i}$ denotes the Riemann-Liouville fractional integral operator. Hence, \eqref{RicVol} and \eqref{E:phichi-1} reduce to the following  system of fractional Riccati equations,
\begin{align*}
\phi' &= \psi b^0  + \frac{1}{2}  \psi A^0 \psi^\top, && \phi(0)=0,\\
D^{\alpha_i} \psi_i &= f_i +  \psi b^i  + \frac{1}{2}  \psi A^i \psi^\top, \quad i=1,\ldots,d,  &&I^{1-\alpha}\psi(0)=u,
\end{align*}
where $D^{\alpha_i}= \frac {d}{dt} I^{1-\alpha_i}$ is the Riemann-Liouville fractional derivative.  Moreover, for $t=0$, \eqref{eq:expaff} reads
\begin{align*}
\E\left[{\rm e}^{u X_T + (f*X)_T}\right]= \exp \left(\phi(T) + I^{1-\alpha}\psi(T) X_0\right)
\end{align*}
where we write $I^{1-\alpha} \psi=(I^{1-\alpha_1}\psi_1, \ldots, I^{1-\alpha_d}\psi_d)$. This generalizes the expressions in \cite{EER:06,EER:07}. Notice that the identity $L_{\alpha_i}*K_{\alpha_i}\equiv1$ is equivalent to the identity $D^{\alpha_i} (I^{\alpha_i} f)=f$.
\end{example}

\section{The Volterra Ornstein--Uhlenbeck process} \label{S:VOU}

The particular specification of \eqref{a and b} where $A^1=\cdots=A^d=0$, so that $a\equiv A^0$ is a constant symmetric positive semidefinite matrix, yields an affine Volterra process with state space $E=\R^d$ that we call the {\em Volterra Ornstein--Uhlenbeck process}. It is the solution of the equation
\[
X_t = X_0 + \int_0^t K(t-s)(b^0 + BX_s)ds + \int_0^t K(t-s) \sigma dW_s,
\]
where $\sigma\in\R^{d\times d}$ is a constant matrix with $\sigma\sigma^\top=A^0$. Here existence and uniqueness is no issue. Indeed, Lemma~\ref{L:ZX} with $T=t$ yields the explicit formula
\[
X_t = \left(\id - \int_0^t R_B(s)ds\right) X_0 + \left( \int_0^t E_B(s)ds\right) b^0  + \int_0^t E_B(t-s)\sigma dW_s,
\]
where $R_B$ is the resolvent of $-KB$ and $E_B=K - R_B*K$. In particular $X_t$ is Gaussian. Furthermore, the solution of the Riccati--Volterra equation~\eqref{RicVol} is obtained explicitly via Lemma~\ref{L:RicVol_res} as
\[
\psi = uE_B + f*E_B.
\]
The quadratic variation of the process $Y$ in \eqref{eq:dY} is given by
\[
\langle Y\rangle_t = \int_0^t \psi(T-s) \sigma\sigma^\top \psi(T-s)^\top ds,
\]
and is in particular deterministic. The martingale condition in Theorem~\ref{T:cf} is thus clearly satisfied, and the exponential-affine transform formula \eqref{eq:expaff} holds for any $T<\infty$, $u\in(\C^d)^*$, and $f\in L^1([0,T],(\C^d)^*)$.

\section{The Volterra square-root process} \label{S:VSQRT}

We now consider affine Volterra processes whose state space is the nonnegative orthant $E=\R^d_+$. We let $K$ be diagonal with scalar kernels $K_i\in L^2_{\rm loc}(\R_+,\R)$ on the diagonal. The coefficients $a$ and $b$ in \eqref{a and b} are chosen so that $A^0=0$, $A^i$ is zero except for the $(i,i)$ element which is equal to $\sigma_i^2$ for some $\sigma_i>0$, and
\begin{equation} \label{sqrt2}
\text{$b^0\in\R^d_+$ and $B_{ij}\ge0$ for $i\ne j$.}
\end{equation}
The conditions on $a$ and $b$ are the same as in the classical situation $K\equiv\id$, in which case they are necessary and sufficient for \eqref{SVE} to admit an $\R^d_+$-valued solution for every initial condition $X_0\in\R^d_+$. With this setup, we obtain an affine Volterra process that we call the {\em Volterra square-root process}. It is the solution of the equation
\begin{equation} \label{VolSqrt}
X_{i,t} = X_{i,0} + \int_0^t K_i(t-s) b_i(X_s) ds + \int_0^t K_i(t-s) \sigma_i \sqrt{X_{i,s}}dW_{i,s}, \quad i=1,\ldots,d.
\end{equation}
The Riccati--Volterra equation \eqref{RicVol} becomes
\begin{equation} \label{RicVolSqrt}
\psi_i(t) = u_i K_i(t) + \int_0^t K_i(t-s) \left( f_i(s) + \psi(s)b^i + \frac{\sigma_i^2}{2} \psi_i(s)^2 \right)  ds, \quad i=1,\ldots,d.
\end{equation}
The following theorem is our main result on Volterra square-root processes.

\begin{theorem} \label{T:VolSqrt}
Assume each $K_i$ satisfies \eqref{K_gamma} and the shifted kernels $\Delta_h K_i$ satisfy \eqref{eq:K orthant} for all $h\in[0,1]$. Assume also that \eqref{sqrt2} holds.
\begin{enumerate}
\item\label{T:VolSqrt:1} The stochastic Volterra equation \eqref{VolSqrt} has a unique in law $\R^d_+$-valued continuous weak solution $X$ for any initial condition $X_0\in\R^d_+$. For each $i$, the paths of $X_i$ are H\"older continuous of any order less than $\gamma_i/2$, where $\gamma_i$ is the constant associated with $K_i$ in \eqref{K_gamma}.

\item\label{T:VolSqrt:2} For any $u\in(\C^d)^*$ and $f\in L^1_{\rm loc}(\R_+,(\C^d)^*))$ such that
\[
\text{${\rm Re\,} u_i \le0$ and ${\rm Re\,} f_i\le0$ for all $i=1,\ldots,d$,}
\]
the Riccati--Volterra equation \eqref{RicVolSqrt} has a unique global solution $\psi\in L^2_{\rm loc}(\R_+,(\C^d)^*)$, which satisfies ${\rm Re\,}\psi_i\le0$, $i=1,\ldots,d$. Moreover, the exponential-affine transform formula \eqref{eq:expaff} holds with $Y$ given by \eqref{eq:dY}--\eqref{Y}.
\end{enumerate}
\end{theorem}

\begin{example}
A sufficient condition for $K_i$ to satisfy the assumptions of Theorem~\ref{T:VolSqrt} is that it satisfies \eqref{K_gamma} and is completely monotone and not identically zero; see Example~\ref{E:CM}. This covers in particular the gamma kernel $t^{\alpha-1}{\rm e}^{-\beta t}$ with $\alpha\in(\frac12,1]$ and $\beta\ge0$.
\end{example}

\begin{proof}
Thanks to \eqref{sqrt2} and the form of $\sigma(x)$, Theorem~\ref{T:existence orthant} yields an $\R^d_+$-valued continuous weak solution $X$ of \eqref{VolSqrt} for any initial condition $X_0\in\R^d_+$. The stated path regularity then follows from the last statement of Lemma~\ref{L:Holder_bound}.

Next, the existence, uniqueness, and non-positivity statement for the Riccati--Volterra equation \eqref{RicVolSqrt} is proved in Lemma~\ref{L:sqrtRic} below. Thus in order to apply Theorem~\ref{T:cf} to obtain the exponential-affine transform formula, it suffices to argue that ${\rm Re\,}Y_t$ is bounded above on $[0,T]$, since $\exp(Y)$ is then bounded and hence a martingale. This is done using Theorem~\ref{T:aff_past}, and we start by observing that 
\[
\pi^{\bf r}_{h,i}(t) = - \int_0^h \psi^{\bf r}_i(h-s)L_i(t+ds), \quad t\ge0,
\]
where $\pi_h = \Delta_h\psi *L - \Delta_h(\psi*L)$ and we write $\pi^{\bf r}_h={\rm Re\,}\pi_h$ and $\psi^{\bf r}={\rm Re\,}\psi$. Due to the assumption \eqref{eq:K orthant} on $L_i$ and since $-\psi^{\bf r}_i\ge0$, it follows that $\pi^{\bf r}_{h,i}$ is nonnegative and non-increasing. 

As in the proof of Theorem~\ref{T:existence orthant}, each $K_i$ satisfies \eqref{eq:orth:1} and \eqref{eq:orth:2}. This implies that the total variation bound \eqref{T:aff_past:TVbound} holds, so that Theorem~\ref{T:aff_past}\ref{T:aff_past:2} yields
\[
{\rm Re\,}Y_t  =  {\rm Re\,}\phi(h) + ({\rm Re\,}\Delta_h f *X)_t + (\Delta_h\psi^{\bf r}*L)(0)X_t - \pi^{\bf r}_h(t)X_0 + (d\pi^{\bf r}_h * X)_t
\]
where $h=T-t$ and, since $A^0=0$,
\[
\phi(h)=\int_0^h \psi(s)b^0 ds.
\]
Observe that $\psi^{\bf r}$, $(\Delta_h\psi^{\bf r}*L)(0)$, ${\rm Re\,}\Delta_h f$, $-\pi^{\bf r}_h$, and $d\pi^{\bf r}_h$ all have nonpositive components. Since $b^0$ and $X$ take values in $\R^d_+$ we thus get
\[
{\rm Re\,}Y_t \le 0.
\]
Thus $\exp(Y)$ is bounded, whence Theorem~\ref{T:cf} is applicable and the exponential-affine transform formula holds.

It remains to prove uniqueness in law for $X$. The law of $X$ is determined by the Laplace transforms $\E[\exp(-\sum_{i=1}^n\lambda_i X_{t_i})]$ with $n\in\N$, $\lambda_i\in(\R^d)^*$ with nonnegative components, and $t_i\ge0$. Uniqueness thus follows since these Laplace transforms are approximated by the quantities $\E[\exp( (f*X)_T)]$ as $f$ ranges through all $(\R^d)^*$-valued continuous functions $f$ with nonpositive components, and $T$ ranges through~$\R_+$.
%It remains to prove uniqueness in law for $X$. This follows since the law of $X$ is determined by the Laplace functionals $\E[\exp( (f*X)_T)]$ as $f$ ranges through, say, all $(\R^d)^*$-valued continuous functions $f$ with nonpositive components, and $T$ ranges through~$\R_+$.
\end{proof}

\begin{lemma} \label{L:sqrtRic}
Assume $K$ is as in Theorem~\ref{T:VolSqrt}. Let $u\in(\C^d)^*$ and $f\in L^1_{\rm loc}(\R_+,(\C^d)^*))$ satisfy
\[
\text{${\rm Re\,} u_i \le0$ and ${\rm Re\,} f_i\le0$ for all $i=1,\ldots,d$.}
\]
Then the Riccati--Volterra equation \eqref{RicVolSqrt} has a unique global solution $\psi\in L^2_{\rm loc}(\R_+,(\C^d)^*)$, and this solution satisfies ${\rm Re\,}\psi_i\le0$, $i=1,\ldots,d$.
\end{lemma}

\begin{proof}
By Theorem~\ref{T:local Volterra} there exists a unique non-continuable solution $(\psi,T_{\rm max})$ of \eqref{RicVolSqrt}.  Let $\psi^{\bf r}$ and $\psi^{\bf i}$ denote the real and imaginary parts of $\psi$. They satisfy the equations
\begin{align*}
\psi^{\bf r}_i &= ({\rm Re\,}u_i) K_i + K_i * \left( {\rm Re\,}f_i + \psi^{\bf r}b^i  + \frac{\sigma_i^2}{2} \left( (\psi^{\bf r}_i)^2 - (\psi^{\bf i}_i)^2\right)\right) \\
\psi^{\bf i}_i &= ({\rm Im\,}u_i) K_i + K_i * \left( {\rm Im\,}f_i + \psi^{\bf i}b^i  + \sigma_i^2 \psi^{\bf r}_i \psi^{\bf i}_i\right)
\end{align*}
on $[0,T_{\rm max})$. Moreover, on this interval, $-\psi^{\bf r}_i$ satisfies the linear equation
\[
\chi_i = -({\rm Re\,}u_i) K_i + K_i * \left( -{\rm Re\,}f_i + \chi b^i  + \frac{\sigma_i^2}{2} \left( (\psi^{\bf i}_i)^2 + \chi_i \psi^{\bf r}_i\right)\right).
\]
Due to \eqref{sqrt2} and since ${\rm Re\,}u$ and ${\rm Re\,}f$ both have nonpositive components, Theorem~\ref{T:positive_2} yields $\psi^{\bf r}_i\le0$, $i=1,\ldots,d$. Next, let $g\in L^2_{\rm loc}([0,T_{\rm max}),(\R^d)^*)$ and $h,\ell \in L^2_{\rm loc}(\R_+,(\R^d)^*)$ be the unique solutions of the linear equations
\begin{align*}
g_i &= |{\rm Im\,}u_i| K_i + K_i * \left( |{\rm Im\,}f_i| + gb^i + \sigma_i^2\psi^{\bf r}_i g_i \right) \\
h_i &= |{\rm Im\,}u_i| K_i + K_i * \left( |{\rm Im\,}f_i| + hb^i \right) \\
\ell_i &= ({\rm Re\,}u_i) K_i + K_i * \left( {\rm Re\,}f_i + \ell b^i - \frac{\sigma_i^2}{2}h_i^2 \right).
\end{align*}
These solutions exist on $[0,T_{max})$ thanks to Corollary~\ref{C:linearVE}. We now perform multiple applications of Theorem~\ref{T:positive_2}. The functions $g\pm \psi^{\bf i}$ satisfy the equations
\[
\chi_i = 2({\rm Im\,}u_i)^{\pm} K_i + K_i * \left( 2({\rm Im\,}f_i)^{\pm} + \chi b^i + \sigma_i^2\psi^{\bf r}_i \chi_i \right)
\]
on $[0,T_{\rm max})$, so $|\psi^{\bf i}_i| \le g_i$ on $[0,T_{\rm max})$ for all $i$. Similarly, $h-g$ satisfies the equation
\[
\chi_i = K_i*\left( \chi b^i - \sigma_i^2 \psi^{\bf r}_i g_i \right)
\]
on $[0,T_{\rm max})$, so $g_i\le h_i$ on $[0,T_{\rm max})$. Finally, $\psi^{\bf r} - \ell$ satisfies the equation
\[
\chi_i = K_i * \left( \chi b^i + \frac{\sigma_i^2}{2} \left( (\psi^{\bf r}_i)^2 + h_i^2 - (\psi^{\bf i}_i)^2 \right) \right),
\]
on $[0,T_{\rm max})$, so $\ell_i\le \psi^{\bf r}_i$ on $[0,T_{\rm max})$. In summary, we have shown that
\[
\text{$\ell_i \le \psi^{\bf r}_i \le 0$ and $|\psi^{\bf i}_i| \le h_i$ on $[0,T_{\rm max})$ for $i=1,\ldots,d$.}
\]
Since $\ell$ and $h$ are global solutions and thus have finite norm on any bounded interval, this implies that $T_{\rm max}=\infty$ and completes the proof of the lemma.
\end{proof}

\section{The Volterra Heston model} \label{S:VH}
We now consider an affine Volterra process with state space $\R\times\R_+$, which can be viewed as a generalization of the classical \cite{heston1993closed} stochastic volatility model in finance, and which we refer to as the {\em Volterra Heston model}. We thus take $d=2$ and consider the process $X=(\log S, V)$, where the price process $S$ and its variance process $V$ are given by
\begin{equation}\label{E:HestonlogS}
\frac{dS_t}{S_t} =  \sqrt{V_t} \,\big( \sqrt{1 - \rho^2} \,dW_{1,s} +  \rho \,dW_{2,s}\big), \qquad S_0\in(0,\infty), 
\end{equation}
and
\begin{equation}\label{E:HestonV}
V_t = V_0 +  \int_0^t K(t-s) \left( \kappa(\theta - V_s)ds +  \sigma \sqrt{V_s}\,dW_{2,s}\right),
\end{equation}
with kernel $K \in L^2_{\rm loc}(\R_+,\R)$, a standard Brownian motion $W=(W_1,W_2)$, and parameters $V_0,\kappa,\theta,\sigma \in\R_+ $ and $\rho \in [-1,1]$. Here the notation has been adapted to comply with established conventions in finance. Weak existence and uniqueness of $V$ follows from Theorem~\ref{T:VolSqrt} under suitable conditions on $K$. This in turn determines $S$. Moreover, observe that the log-price satisfies
\[
\log S_t = \log S_0 -\int_0^t \frac{V_s}{2} \,ds +  \int_0^t \sqrt{V_s}\,\big( \sqrt{1 - \rho^2} \,dW_{1,s} +  \rho \,dW_{2,s}\big).
\]
Therefore the process $X=(\log S,V)$ is indeed an affine Volterra process with diagonal kernel $\mbox{diag}(1,K)$ and coefficients $a$ and $b$ in \eqref{a and b} given by
\begin{align*}
A^0 &=A^1=0, \quad A^2= \left( \begin{array}{cc} 1 & \rho \sigma    \\  \rho \sigma  & \sigma^2  \end{array} \right), \nonumber \\
b^0 &= \left( \begin{array}{c} 0     \\   \kappa \theta \end{array} \right),\quad B=\left( \begin{array}{cc} 0 & -\frac12      \\  0 & -\kappa\end{array} \right).
\end{align*}
The Riccati--Volterra equation \eqref{RicVol} takes the form
\begin{align}
\psi_1 &= u_1+1*f_1,\label{E:RicHeston0} \\
\psi_2 &=u_2K +K*\left(f_2+\frac12\left( \psi_1^2-\psi_1\right) - \kappa  \psi_2 + \frac{1}{2}\left(\sigma^2 \psi^2_2 + 2 \rho \sigma \psi_1\psi_2\right)\right). \label{E:RicHeston1}
\end{align}

\begin{theorem} \label{T:VolterraHeston}
Assume $K$ satisfies \eqref{K_gamma} and the shifted kernels $\Delta_h K$ satisfy \eqref{eq:K orthant} for all $h\in[0,1]$.
\begin{enumerate}
\item\label{T:VolHeston:1} The stochastic Volterra equation \eqref{E:HestonlogS}-\eqref{E:HestonV} has a unique in law $\R\times \R_+$-valued continuous weak solution $(\log S,V)$ for any initial condition $(\log S_0,V_0) \in\R\times\R_+$. The paths of $V$ are H\"older continuous of any order less than $\gamma/2$, where $\gamma$ is the constant associated with $K$ in \eqref{K_gamma}.

\item\label{T:VolHeston:2} Let $u\in(\C^2)^*$ and $f\in L^1_{\rm loc}(\R_+,(\C^2)^*))$ be such that
\[
\text{${\rm Re\,} \psi_1 \in[0,1]$, ${\rm Re\,} u_2 \le0$ and ${\rm Re\,} f_2\le0$.}
\]
where $\psi_1$ is given by \eqref{E:RicHeston0}. Then the Riccati--Volterra equation~\eqref{E:RicHeston1} has a unique global solution $\psi_2\in L^2_{\rm loc}(\R_+,\C^*)$, which satisfies ${\rm Re\,}\psi_2\le0$. Moreover, the exponential-affine transform formula \eqref{eq:expaff} holds with $Y$ given by \eqref{eq:dY}--\eqref{Y}.

\item\label{T:VolHeston:3} The process $S$ is a martingale.
\end{enumerate}
\end{theorem}

\begin{proof}
As already mentioned above, part~\ref{T:VolHeston:1} follows directly from Theorem~\ref{T:VolSqrt} along with the fact that $S$ is determined by $V$. Part~\ref{T:VolHeston:3} is proved in Lemma~\ref{T:Hestonmgle_new} below. The existence, uniqueness, and non-positivity statement for the Riccati--Volterra equation \eqref{E:RicHeston1} is proved in Lemma~\ref{L:HestonRic} below. Thus in order to apply Theorem~\ref{T:cf} to obtain the exponential-affine transform formula, it suffices to argue that $\exp(Y)$ is a martingale. This is done using Theorem~\ref{T:aff_past} and part~\ref{T:VolHeston:3}. As the argument closely parallels that of the proof of Theorem~\ref{T:VolSqrt}, we only provide an outline. We use the notation of Theorem~\ref{T:aff_past} and Theorem~\ref{T:VolSqrt}, in particular $\pi_h$ and $\pi_h^{\bf r}={\rm Re\,}\pi_h$, and let $L$ be the resolvent of the first kind of $K$. Theorem~\ref{T:aff_past} is applicable and gives
\begin{align}
{\rm Re\,}Y_t  &= \psi_1^{\bf r}(h) \log S_t+({\rm Re\,}\Delta_h f_1 *\log S)_t+{\rm Re\,}\phi(h) +(\Delta_h\psi_2^{\bf r}*L)(0)V_t \notag \\
&\quad + ({\rm Re\,}\Delta_h f_2 *V)_t - \pi^{\bf r}_{h,2}(t)V_0 + (d\pi^{\bf r}_{h,2} * V)_t\label{eq:VoltHestonaffinepast}
\end{align}
where $h=T-t$ and
\[
\phi(h)=\kappa\theta\int_0^h\psi_2(s)\,ds.
\]
Since $\psi_1^{\bf r} \in[0,1]$, integration by parts yields
\begin{align*}
\psi^{\bf r}_1(h)\log S_t+({\rm Re\,}\Delta_h f_1 *\log S)_t&=\psi_1^{\bf r}(T)\log S_0+\int_0^t \psi_1^{\bf r}(T-s)\,d\log S_s\\
&\leq \psi_1^{\bf r}(T)\log S_0+U_t - \frac12\langle U\rangle_t,
\end{align*}
where
\[
U_t= \int_0^t \psi_1^{\bf r}(T-s)\sqrt{V_s}\,\big( \sqrt{1 - \rho^2} \,dW_{1,s} +  \rho \,dW_{2,s}\big).
\]
This observation and inspection of signs and monotonicity properties applied to~\eqref{eq:VoltHestonaffinepast} show that 
\[
|\exp(Y_t)|=\exp({\rm Re\,}Y_t)\leq S_0^{\psi_1^{\bf r}(T)}\exp(U_t - \frac12\langle U\rangle_t),
\]
where the right-hand side is a true martingale by Lemma \ref{T:Hestonmgle_new}. Thus $\exp(Y)$ is a true martingale, Theorem~\ref{T:cf} is applicable, and the exponential-affine transform formula holds.
\end{proof}

\begin{example}[Rough Heston model] In the fractional case $K(t)=\frac{t^{1-\alpha}}{\Gamma(\alpha)}$ with $\alpha \in (\frac 12, 1)$ we recover the rough Heston model introduced and studied by \cite{EER:06,EER:07}.  Theorem~\ref{T:VolterraHeston} generalizes some of their main results. For instance, with the notation of Example~\ref{E:faffine} and using that $L(dt)= \frac{t^{-\alpha}}{\Gamma(1-\alpha)}dt$, we have
\[
\chi = ( \psi_1 , I^{1-\alpha} \psi_2 ),
\]
which yields the full Fourier--Laplace functional with integrated log-price and variance,
\[
\E\left[{\rm e}^{ u_1 \log S_T + u_2 V_T + (f_1*\log S)_T+(f_2 * V)_T}\right] = \exp\left( \phi(T) +  \psi_1(T) \log S_0 + I^{1-\alpha} \psi_2(T) V_0    \right),
\]
where $\psi_1$ is given by~\eqref{E:RicHeston0}, and $\phi$ and $\psi_2$ solve the fractional Riccati equations
\begin{align*}
\phi' &=  \kappa \theta \psi_2, &&\phi(0)=0,\\
D^{\alpha}\psi_2 &= f_2+\frac12\left( \psi_1^2-\psi_1\right) + ( \rho \sigma \psi_1- \kappa )  \psi_2 + \frac{\sigma^2}{2} \psi^2_2,  &&I^{1-\alpha}\psi_2 (0)=u_2. 
\end{align*}
\end{example}

We now proceed with the lemmas used in the proof of Theorem~\ref{T:VolterraHeston}.

\begin{lemma}\label{T:Hestonmgle_new}
Let $g\in L^\infty(\R_+,\R)$ and define 
\[
U_t = \int_0^t g(s) \sqrt{V_s}(\sqrt{1-\rho^2}dW_{1,s} + \rho dW_{2,s})).
\] 
Then the stochastic exponential $\exp(U_t - \frac12\langle U\rangle_t)$ is a martingale. In particular, $S$ is a martingale.
\end{lemma}

\begin{proof}
Define $M_t=\exp(U_t - \frac12\langle U\rangle_t)$. Since $M$ is a nonnegative local martingale, it is a supermartingale by Fatou's lemma, and it suffices to show that $\E[M_T]\ge 1$ for any $T\in\R_+$. To this end, define stopping times $\tau_n=\inf\{t\ge0\colon V_t> n\}\wedge T$. Then $M^{\tau_n}$ is a uniformly integrable martingale for each $n$ by Novikov's condition, and we may define probability measures $\Q^n$ by
\[
\frac{d\Q^n}{d\P} = M_{\tau_n}.
\]
By Girsanov's theorem, the process $dW^n_t = dW_{2,t} - \bm1_{\{t\le \tau_n\}}\rho\,g(t)\sqrt{V_t}dt$ is Brownian motion under $\Q^n$, and we have
\[
V = V_0 + K*( (\kappa\theta-(\kappa-\rho\sigma g\bm1_{\lc0,\tau_n\rc})V)dt + \sigma\sqrt{V}dW^n).
\]
Let $\gamma$ be the constant appearing in \eqref{K_gamma} and choose $p>2$ sufficiently large that $\gamma/2-1/p>0$. Observe that the expression $\kappa\theta-(\kappa-\rho\sigma g(t)\bm1_{\{t\le\tau_n(\omega)\}})v$ satisfies a linear growth condition in $v$, uniformly in $(t,\omega)$. Therefore, due to Lemma~\ref{L:moment bound} and Remark~\ref{R:moment bound}, we have the moment bound
\[
\sup_{t\le T} \E_{\Q^n}[ |V_t|^p ] \le c
\]
for some constant $c$ that does not depend on $n$. For any real-valued function $f$, write
\[
|f|_{C^{0,\alpha}(0,T)} = \sup_{0\le s<t\le T} \frac{|f(t)-f(s)|}{|t-s|^\alpha}
\]
for its $\alpha$-H\"older seminorm. We then get
\begin{align*}
\Q^n(\tau_n< T) &\le \Q^n\Big( \sup_{t\le T} V_t> n\Big) \\
&\le \Q^n\Big( V_0 + |V|_{C^{0,0}(0,T)} > n\Big) \\
&\le \left(\frac{1}{n-V_0}\right)^p \E_{\Q^n}\left[ |V|_{C^{0,0}(0,T)}^p \right] \\
&\le \left(\frac{1}{n-V_0}\right)^p c'
\end{align*}
for a constant $c'$ that does not depend on $n$, using Lemma~\ref{L:Holder_bound} with $\alpha=0$ for the last inequality. We deduce that
\[
\E_\P\left[M_T\right] \ge \E_\P\left[M_T\bm 1_{\{\tau_n=T\}}\right] = \Q^n(\tau_n=T) \ge 1-\left(\frac{1}{n-V_0}\right)^p c',
\]
and sending $n$ to infinity yields $\E_\P[M_T]\ge 1$. This completes the proof.
\end{proof}

\begin{lemma} \label{L:HestonRic}
Assume $K$ is as in Theorem~\ref{T:VolterraHeston}. Let $u\in(\C^2)^*$ and $f\in L^1_{\rm loc}(\R_+,(\C^2)^*))$ be such that
\[
\text{${\rm Re\,} \psi_1 \in[0,1]$, ${\rm Re\,} u_2 \le0$ and ${\rm Re\,} f_2\le0$,}
\]
with $\psi_1$ given by \eqref{E:RicHeston0}. Then the Riccati--Volterra equation~\eqref{E:RicHeston1} has a unique global solution $\psi_2\in L^2_{\rm loc}(\R_+,\C^*)$, which satisfies ${\rm Re\,}\psi_2\le0$.
\end{lemma}

\begin{proof}
The proof parallels that of Lemma~\ref{L:sqrtRic}. For any complex number $z$, we denote by $z^{\bf r}$ and $z^{\bf i}$ the real and imaginary parts of $z$. We rewrite equation~\eqref{E:RicHeston1} for $\psi_2$ as
\begin{equation}\label{E:RicHeston2}
\psi_2=u_2K+K\ast\left(f_2+\frac{1}{2}( \psi_1^2 -\psi_1)+ ( \rho \sigma \psi_1   -\kappa)\psi_2  + \frac{\sigma^2}{2} \psi^2_2\right).
\end{equation}
By Theorem~\ref{T:local Volterra} there exists a unique non-continuable solution $(\psi_2,T_{\rm max})$ of \eqref{E:RicHeston2}. The functions $\psi_2^{\bf r}$ and $\psi_2^{\bf i}$ satisfy the equations 
\begin{align*}
\psi_2^{\bf r} &= u_2^{\bf r}K + K\ast \bigg(f_2^{\bf r} +\frac{1}{2}( (\psi_1^{\bf r})^2 -\psi_1^{\bf r}- (\psi_1^{\bf i})^2)-\rho\sigma \psi_1^{\bf i}\psi_2^{\bf i} \\
&\qquad\qquad\qquad\qquad -\frac{\sigma^2}{2}(\psi_2^{\bf i})^2 +(\rho\sigma \psi_1^{\bf r}-\kappa)\psi_2^{\bf r}+\frac{\sigma^2}{2}(\psi_2^{\bf r})^2\bigg) \\
\psi_2^{\bf i} &= u_2^{\bf i} K+K\ast\left( f_2^{\bf i}+\frac{1}{2}\left(2 \psi_1^{\bf r} \psi_1^{\bf i}- \psi_1^{\bf i}\right)+\rho\sigma \psi_1^{\bf i}\psi_2^{\bf r}+(\rho\sigma \psi_1^{\bf r}-\kappa+\sigma^2\psi_2^{\bf r})\psi_2^{\bf i}\right)
\end{align*}
on $[0,T_{\rm max})$. After some rewriting, we find that on $[0,T_{\rm max})$, $-\psi^{\bf r}_2$ satisfies the linear equation
\begin{align*}
\chi = -u_2^{\bf r} K + K * \bigg( &-f_2^{\bf r}+\frac{1}{2}( \psi_1^{\bf r}-(\psi_1^{\bf r})^2 +(1-\rho^2)(\psi_1^{\bf i})^2) \\
&\quad +\frac{(\sigma\psi_2^{\bf i}+\rho \psi_1^{\bf i})^2}{2}-\left(\rho\sigma \psi_1^{\bf r}-\kappa+\frac{\sigma^2}{2}\psi_2^{\bf r}\right)\chi\bigg).
\end{align*}
Due to \eqref{sqrt2} and since $\psi_1^{\bf r},|\rho|\in[0,1]$, and $f^{\bf r}_2$ and $u_2^{\bf r}$ are nonpositive, Theorem~\ref{T:positive_2} yields $\psi^{\bf r}_2\le0$ on $[0,T_{\rm max})$. 

Now, if $\sigma=0$, then \eqref{E:RicHeston2} is a linear Volterra equation and thus admits a unique global solution $\psi_2\in L^2_{\rm loc}(\R_+,\C^*)$ by Corollary~\ref{C:linearVE}. Therefore it suffices to consider the case $\sigma>0$. 

Following the proof of Lemma~\ref{L:sqrtRic}, we let $g\in L^2_{\rm loc}([0,T_{\rm max}),(\R)^*)$ and $h,\ell \in L^2_{\rm loc}(\R_+,(\R)^*)$ be the unique solutions of the linear equations 
\begin{align*}
g&= |u_2^{\bf i}| K+\left|\rho\sigma^{-1}u_1^{\bf i}\right|+K\ast\bigg( \left|\rho\sigma^{-1}(L\ast f_1^{\bf i})+f_2^{\bf i}+\frac{\psi_1^{\bf i}}{2}\left(2(1-\rho^2)\psi_1^{\bf r}-1+\frac{2\kappa\rho}{\sigma}\right)\right| \\
&\quad\qquad\qquad\qquad\qquad\qquad\qquad +(\rho\sigma \psi_1^{\bf r}-\kappa+\sigma^2\psi_2^{\bf r})g\bigg) \\
h &=  |u_2^{\bf i}| K+\left|\rho\sigma^{-1}u_1^{\bf i}\right|+K\ast\bigg( \left|\rho\sigma^{-1}(L\ast f_1^{\bf i})+ f_2^{\bf i}+\frac{\psi_1^{\bf i}}{2}\left(2(1-\rho^2)\psi_1^{\bf r}-1+\frac{2\kappa\rho}{\sigma}\right)\right|\\
&\quad\quad\quad\quad\quad\quad\quad\quad\quad\quad\quad\quad\quad+(\rho\sigma \psi_1^{\bf r}-\kappa)h\bigg)\\
\ell &=u_2^{\bf r}K + K\ast \bigg(f_2^{\bf r}+\frac{1}{2}((\psi_1^{\bf r})^2 - \psi_1^{\bf r}-(\psi_1^{\bf i})^2)-|\rho\sigma \psi_1^{\bf i}| \left(h+\left|\rho \psi_1^{\bf i}\sigma^{-1}\right|\right) \\
&\qquad\qquad\qquad\qquad -\frac{\sigma^2}{2}\left(h+\left|\rho \psi_1^{\bf i}\sigma^{-1}\right|\right)^2+(\rho\sigma \psi_1^{\bf r}-\kappa)\ell\bigg).
\end{align*}
These solutions exist on $[0,T_{\rm max})$ thanks to Corollary~\ref{C:linearVE}. We now perform multiple applications of Theorem~\ref{T:positive_2}. The functions $g\pm  (\psi_2^{\bf i}+(\rho \psi_1^{\bf i}\sigma^{-1}))$ satisfy the equations
\begin{align*}
\chi= 2(u_2^{\bf i})^{\pm} K+2\left(\rho\sigma^{-1}u_1^{\bf i}\right)^{\pm}+K\ast\bigg(& 2\left(\rho\sigma^{-1}(L\ast f_1^{\bf i})+f_2^{\bf i}+\frac{\psi_1^{\bf i}}{2}\left(2(1-\rho^2)\psi_1^{\bf r}-1+\frac{2\kappa\rho}{\sigma}\right)\right)^{\pm} \\
&\quad +(\rho\sigma \psi_1^{\bf r}-\kappa+\sigma^2\psi_2^{\bf r})\chi\bigg)
\end{align*}
on $[0,T_{max})$, so that $0\leq |\psi_2^{\bf i}+\left(\rho \psi_1^{\bf i}\sigma^{-1}\right)|\leq g$ on $[0,T_{max})$. Similarly, the function $h-g$ satisfies the equation
\[
\chi=K*\left(-\sigma^2\psi_2^{\bf r}g+(\rho\sigma \psi_1^{\bf r}-\kappa)\chi\right)
\]
on $[0,T_{\rm max})$, so that $g\le h$ on  $[0,T_{\rm max})$. This yields $|\psi_2^{\bf i}|\le h+\left|\rho \psi_1^{\bf i}\sigma^{-1}\right|$ on  $[0,T_{max})$. Finally, the function $\psi_2^{\bf r}-\ell$ satisfies the linear equation 
\begin{align*}
\chi= K\ast \bigg(&|\rho\sigma \psi_1^{\bf i}| \left(h+\left|\rho \psi_1^{\bf i}\sigma^{-1}\right|\right)-\rho\sigma \psi_1^{\bf i}\psi_2^{\bf i} \\
&+\frac{\sigma^2}{2}\left(\left(h+\left|\rho \psi_1^{\bf i}\sigma^{-1}\right|\right)^2-\left(\psi_2^{\bf i}\right)^2+\left(\psi_2^{\bf r}\right)^2\right)+(\rho\sigma \psi_1^{\bf r}-\kappa)\chi\bigg)
\end{align*}
on $[0,T_{\rm max})$, so that $\ell\le \psi_2^{\bf r}\leq 0$ on  $[0,T_{\rm max})$.  Since $h$ and $\ell$ are global solutions and thus have finite norm on any bounded interval, this implies that $T_{\rm max}=\infty$ and completes the proof of the lemma.\end{proof}

We conclude this section with a remark on an alternative variant of the Volterra Heston model in the spirit of \cite{com_cou_ren_12,GJR:17}.

\begin{example}
Let $\widetilde K$ denotes a scalar locally square integrable non-negative kernel. Consider the following variant of the Volterra Heston model 
\begin{align*}
dS_t &= S_t \sqrt{\widetilde V_t}dB_t, && S_0\in(0,\infty), \\
dV_t &= \kappa (\theta -V_t)dt + \sigma \sqrt{V_t} dB^{\perp}_t, &&V_0\geq 0,\\
\widetilde V_t&= \widetilde V_0 + (\widetilde K*V)_t,
\end{align*}
where $B$ and $B^{\perp}$ are independent Brownian motions. Since  $\widetilde K$ is nonnegative, one readily sees that there exists a unique strong solution taking values in $\R \times \R^2_+$. The $3$-dimensional process $X=(\log S, V,\widetilde V)$ is an affine Volterra process with  
$$ K = \diag (1,1, \widetilde K), \quad  b^0 =\begin{pmatrix}0 \\\kappa \theta\\ 0 \end{pmatrix}, \quad B = \begin{pmatrix}
0 & 0 & - \frac 12  \\
0 &- \kappa  & 0 \\
0 &  1 & 0
\end{pmatrix}, $$ 
$$ A^0 = 0, \quad  A^1 =0, \quad A^2= \diag(0,\sigma^2,0), \quad  A^3= \diag(1,0,0).$$ 
The Riccati--Volterra equation \eqref{RicVol} reads 
\begin{align*}
\psi'_1 &= f_1 , &&\psi_1(0)=u_1, \nonumber \\
\psi_2'&= f_2 +  \psi_3 - \kappa \psi_2  + \frac{\sigma^2}2 \psi_2^2, && \psi_2 (0)=u_2, \\
\psi_3 &= u_3 \widetilde K + \widetilde K*\left(f_3 + \frac 12 \psi_1 (\psi_1 -1)\right). 
\end{align*}
Under suitable conditions the solution exists and is unique, and the process ${\rm e}^{Y}$ with $Y$ given by \eqref{eq:dY}--\eqref{Y} is a true martingale. Hence by Theorem \ref{T:cf} the exponential-affine transform formula \eqref{eq:expaff} holds. We omit the details. In particular, for $f\equiv 0$ we get, using Example~\ref{E:faffine},
\begin{align*}
\chi(t) &=  (\psi * L )(t) = \left(u_1, \psi_2(t), u_3 +   \frac {(u_1^2-u_1)t}2  \right)
\end{align*}
and 
$$\E \left[ {\rm e}^{ u_1 \log S_T + u_2 V_T + u_3 \widetilde V_T} \right] = \exp\left ( \phi(T) +u_1 \log S_0 + \psi_2(T) V_0 +    \left(u_3 +   \frac {(u^2_1-u_1)T}2\right) \widetilde V_0 \right),$$
where $\phi$  and $\psi_2$ solve 
\begin{align*}
\phi' &=  \kappa \theta \psi_2, &&\phi(0)=0,\\
\psi_2' &=  u_3 \widetilde K + \widetilde K*\frac {(u_1^2-u_1)}2 - \kappa \psi_2  + \frac{\sigma^2}2 \psi_2^2,  &&\psi_2 (0)=u_2. 
\end{align*} 
Setting $\widetilde K =  \frac{t^{\alpha-1}}{\Gamma(\alpha)}$  and $u_2=0$, this formula agrees with \citet[Theorem 2.1]{GJR:17}. If $B$ and $B^{\perp}$ are correlated one loses the affine property, as highlighted in \citet[Remark 2.2]{GJR:17}. 
\end{example}

\section*{Acknowledgements}
The authors wish to thank Bruno Bouchard, Omar El Euch, Camille Illand, and Mathieu Rosenbaum for useful comments and fruitful discussions. The authors also thank the anonymous referees for their careful reading of the manuscript and suggestions.

\appendix

\section{Proof of Theorems~\ref{T:sol_lip} and~\ref{T:EXISTENCE}}  \label{S:proof_existence}

\begin{proof}[Proof of Theorem~\ref{T:sol_lip}]
The proof parallels that of \citet[Theorem 2.3]{MARINELLI2010616}, using a contraction mapping principle. Suppose that $p>\max\{2,2/\gamma\}$, with $\gamma$ as in~\eqref{K_gamma}. For $T\ge0$, consider all processes $X$ on $[0,T]$ that satisfy
\[
 \|X\|_{p,T} = \sup_{t \leq T} \E[|X_t|^p]^{1/p} < \infty,
\]
and are continuous in $L^p$ in the sense that $\lim_{s\to t}\E[|X_t-X_s|^p]=0$ for all $t\in[0,T]$. We let $\Hcal_{p,T}$ denote the space of all such $X$, modulo the equivalence relation obtained by identifying processes that are versions of each other. One readily verifies that $(\Hcal_{p,T},\|\fdot\|_{p,T})$ is a Banach space. Thanks to \citet[Proposition 3.21]{PZ07}, every element $X\in\Hcal_{p,T}$ admits a predictable representative, again denoted by $X$. Below we always work with such representatives.

We first prove the existence of a unique solution to~\eqref{SVE} in $\mathcal H_{p,T}$. To this end we consider the following family of norms on $\mathcal H_{p,T}$:
\[
 \|X\|_{p,T,\lambda} :=  \sup_{t \leq T} \E[|e^{-\lambda t} X_t|^p]^{1/p}, \quad \lambda >0.
\]
For every $X\in\Hcal_{p,T}$, define a new process $\Tcal X$ by
\begin{equation*}\label{eq:mapT}
\mathcal T X =  X_0+K*(b(X)dt+\sigma(X)dW).
\end{equation*} 
Lemma~\ref{L:Holder_bound} and the linear growth properties of $b$ and $\sigma$ imply that $\|\mathcal TX\|_{p,T}<\infty$ and $\Tcal X$ is continuous in $L^p$. Thus $\mathcal T X$ lies in $\mathcal H_{p,T}$. Now, since the norms $\|\cdot\|_{p,T}$ and $\|\cdot\|_{p,T,\lambda}$ are equivalent, it is enough to find $\lambda>0$ such that the $\mathcal T$ defines a contraction on $(\mathcal H_{p,T},\|\cdot\|_{p,T,\lambda})$. That is, we look for $\lambda>0$ and $M<1$ such that
\begin{equation}\label{eq:contract temp}
	\| \mathcal T X - \mathcal T Y \|_{p,T,\lambda} \leq M  \| X -  Y \|_{p,T,\lambda}, \quad  X,Y \in \mathcal H_{p,T}.	
\end{equation} 
For $t \leq T$ and $\lambda >0$ we have
\begin{align*}
	 |e^{-\lambda t}((\mathcal T X)_t - (\mathcal T Y)_t)|^p  &\leq  2^{p-1} \left|\int_0^t e^{-\lambda (t-s)}K(t-s)e^{-\lambda s}(b(X_s)-b(Y_s))ds\right|^p\\
	 &\quad  + 2^{p-1} \left|\int_0^t e^{-\lambda (t-s)} K(t-s)e^{-\lambda s}((\sigma(X_s)-\sigma(Y_s))dW_s\right|^p.
	 \end{align*} 
 Arguing as in the proof of Lemma~\ref{L:Holder_bound}, an application of the Jensen and BDG inequalities combined with the Lipschitz property of $b$ and $\sigma$ yields
\begin{equation*}
	 \| \mathcal T X - \mathcal T Y \|_{p,T,\lambda}^p  \leq c \left(\int_0^T e^{-2\lambda u} |K(u)|^2 du \right)^{p/2} \|X-Y\|^p_{p,T,\lambda}
\end{equation*} 
for some constant $c$ depending only on $p$, $T$, and the Lispchitz constant for $b$ and $\sigma$. Since $\int_0^T e^{-2\lambda u} |K(u)|^2 du\to0$ as $\lambda \to \infty$ by the dominated convergence theorem, there exists $\lambda >0$ such that \eqref{eq:contract temp} holds.

%The existence of $\lambda >0$ such that \eqref{eq:contract temp} holds now follows since $\int_0^T e^{-2\lambda u} |K(u)|^2 du\to0$ as $\lambda \to \infty$ due to the dominated convergence theorem.

Let $X$ be the unique fixed point in $\mathcal H_{p,T}$ of the map $\mathcal T$. Lemma~\ref{L:Holder_bound} implies that $X$ has a continuous version. This version is a strong solution of~\eqref{SVE} on $[0,T]$. By virtue of Lemmas~\ref{L:Holder_bound} and~\ref{L:moment bound}, all continuous solutions of~\eqref{SVE} on $[0,T]$ are fixed points in $\mathcal H_{p,T}$ of the map $\mathcal T$, which implies uniqueness. Since $T\ge0$ was arbitrary, it follows that \eqref{SVE} has a unique strong solution on $[0,\infty)$.
\end{proof}

\begin{lemma} \label{L:tight}
Fix an initial condition $X_0\in\R^d$ and a constant $c_{\rm LG}$. Let $\Xcal$ denote the set of all continuous processes $X$ that solve \eqref{SVE} for some continuous coefficients $b$ and $\sigma$ satisfying the linear growth bound~\eqref{LG} with the given constant $c_{\rm LG}$. Then $\Xcal$ is tight, meaning that the family $\{\text{law of $X$} \colon X\in\Xcal\}$ of laws on $C(\R_+;\R^d)$ is tight.
\end{lemma}

\begin{proof}
Let $X\in\Xcal$ be any solution of \eqref{SVE} for some continuous $b$ and $\sigma$ satisfying the linear growth bound~\eqref{LG}. Lemma~\ref{L:moment bound} implies that  $\sup_{u\le T} \E[  |b(X_u)|^p ]$ and $\sup_{u\le T} \E[  |\sigma(X_u)|^p ]$ are bounded above by a constant that only depends on $|X_0|$, $\| K\|_{L^2(0,T)}$, $c_{\rm LG}$, $p$, and $T$. Therefore, since the components of $K$ satisfy~\eqref{K_gamma}, we may apply Lemma~\ref{L:Holder_bound} to obtain
\[
\E\left[ \left( \sup_{0\le s<t\le T} \frac{|X_t-X_s|}{|t-s|^\alpha} \right)^p \right] \le c
\]
for all $\alpha\in[0,\bar\gamma/2-1/p)$, where $\bar\gamma$ is the smallest of the constants $\gamma$ appearing in~\eqref{K_gamma} for the components of $K$, and where $c$ is a constant that only depends on $|X_0|$, $\| K\|_{L^2(0,T)}$, $c_{\rm LG}$, $p$, and $T$, but not on $s$ or $t$, nor on the specific choice of $X\in\Xcal$. Choosing $p>2$ so that $\bar\gamma p/2>1$, and using that closed H\"older balls are compact in $C([0,T];\R^d)$ by the Arzel\`a--Ascoli theorem, it follows that $\Xcal$ is tight in $C(\R_+;\R^d)$.
\end{proof}

\begin{lemma} \label{L:convergence}
Assume that $K$ admits a resolvent of the first kind~$L$. For each $n\in\N$, let $X^n$ be a weak solution of \eqref{SVE} with $b$ and $\sigma$ replaced by some continuous coefficients $b^n$ and $\sigma^n$ that satisfy \eqref{LG} with a common constant $c_{\rm LG}$. Assume that $b^n\to b$ and $\sigma^n\to\sigma$ locally uniformly for some coefficients $b$ and $\sigma$, and that $X^n\Rightarrow X$ for some continuous process $X$. Then $X$ is a weak solution of \eqref{SVE}.
\end{lemma}

\begin{proof}
Lemma~\ref{L:ZX} yields the identity
\[
L*(X^n-X_0) = \int b^n(X^n)dt + \int \sigma^n(X^n) dW.
\]
Moreover, \citet[Theorem~3.6.1(ii) and Corollary~3.6.2(iii)]{GLS:90} imply that the map
\[
F \mapsto L*(F-F(0))
\]
is continuous from $C(\R_+;\R^d)$ to itself. Using also the locally uniform convergence of $b^n$ and $\sigma^n$, the continuous mapping theorem shows that the martingales
\[
M^n = \int \sigma^n(X^n) dW = L*(X^n-X_0) - \int b^n(X^n)dt
\]
converge weakly to some limit $M$, that the quadratic variations $\langle M^n\rangle = \int \sigma^n\sigma^{n\top}(X^n) dt$ converge weakly to $\int \sigma\sigma^\top(X)dt$, and that $\int b^n(X^n)dt$ converge weakly to $\int b(X)dt$.

Consider any $s<t$, $m\in\N$, any bounded continuous function $f\colon\R^m\to\R$, and any $0\le t_1\le\cdots\le t_m\le s$. Observe that the moment bound in Lemma~\ref{L:moment bound} is uniform in $n$ since the $X^n$ satisfy the linear growth condition \eqref{LG} with a common constant. Using \citet[Theorem~3.5]{B:99}, one then readily shows that
\[
\E[ f(X_{t_1},\ldots,X_{t_m}) (M_t-M_s) ] = \lim_{n\to\infty} \E[ f(X^n_{t_1},\ldots,X^n_{t_m}) (M^n_t-M^n_s) ] = 0,
\]
and similarly for the increments of $M_i^n M_j^n-\langle M_i^n,M_j^n\rangle$. It follows that $M$ is a martingale with respect to the filtration generated by $X$ with quadratic variation $\langle M\rangle=\int \sigma\sigma^\top(X)dt$. This carries over to the usual augmentation. Enlarging the probability space if necessary, we may now construct a $d$-dimensional Brownian motion $\overline W$ such that $M=\int\sigma(X)d\overline W$.

The above shows that $L*(X-X_0) = \int b(X)dt + \int\sigma(X)d\overline W$. The converse direction of Lemma~\ref{L:ZX} then yields $X=X_0+K*(b(X)dt+\sigma(X)d\overline W)$, that is, $X$ solves \eqref{SVE} with the Brownian motion $\overline W$.
\end{proof}

\begin{proof}[Proof of Theorem~\ref{T:EXISTENCE}]
Using \citet[Proposition~1.1]{HS:12} we choose Lipschitz coefficients $b^n$ and $\sigma^n$ that satisfy the linear growth bound~\eqref{LG} with $c_{\rm LG}$ replaced by $2c_{\rm LG}$, and converge locally uniformly to $b$ and $\sigma$ as $n\to\infty$. Let $X^n$ be the unique continuous strong solution of \eqref{SVE} with $b$ and $\sigma$ replaced by $b^n$ and $\sigma^n$; see Theorem~\ref{T:sol_lip}. Due to Lemma~\ref{L:tight} the sequence $\{X^n\}$ is tight, so after passing to a subsequence we have $X^n\Rightarrow X$ for some continuous process $X$. The result now follows from Lemma~\ref{L:convergence}.
\end{proof}

\section{Local solutions of Volterra integral equations}

Fix a kernel $K\in L^2_{\rm loc}(\R_+,\R^{d\times d})$ along with functions $g\colon\R_+\to\C^d$ and $p\colon \R_+\times\C^d\to\C^d$, and consider the Volterra integral equation
\begin{equation} \label{eq:VolIntEq}
\psi = g + K*p(\fdot,\psi).
\end{equation}
A {\em non-continuable solution} of \eqref{eq:VolIntEq} is a pair $(\psi,T_{\rm max})$ with $T_{\rm max}\in(0,\infty]$ and $\psi\in L^2_{\rm loc}([0,T_{\rm max}),\C^d)$, such that $\psi$ satisfies \eqref{eq:VolIntEq} on $[0,T_{\rm max})$ and $\|\psi\|_{L^2(0,T_{\rm max})}=\infty$ if $T_{\rm max}<\infty$. If $T_{\rm max}=\infty$ we call $\psi$ a {\em global solution} of \eqref{eq:VolIntEq}. With some abuse of terminology we call a non-continuable solution $(\psi,T_{\rm max})$ {\em unique} if for any $T\in\R_+$ and $\widetilde\psi\in L^2([0,T],\C^d)$ satisfying \eqref{eq:VolIntEq} on $[0,T]$, we have $T<T_{\rm max}$ and $\widetilde\psi=\psi$ on $[0,T]$.

\begin{theorem} \label{T:local Volterra}
Assume that $g\in L^2_{\rm loc}(\R_+,\C^d)$, $p(\fdot,0)\in L^1_{\rm loc}(\R_+,\C^d)$, and that for all $T\in\R_+$ there exist a positive constant $\Theta_T$ and a function $\Pi_T\in L^2([0,T],\R_+)$ such that
\begin{equation}\label{eq:local Volterra:1}
|p(t,x)-p(t,y)| \le \Pi_T(t)|x-y|+\Theta_T |x-y|(|x|+|y|), \qquad x,y\in\C^d,\ t\le T.
\end{equation}
The Volterra integral equation \eqref{eq:VolIntEq} has a unique non-continuable solution $(\psi,T_{\rm max})$. If $g$ and $p$ are real-valued, then so is $\psi$.
\end{theorem}

\begin{remark}
If $K\in L^{2+\varepsilon}_{\rm loc}$ for some $\varepsilon>0$, then it is possible to apply \citet[Theorem~12.4.4]{GLS:90} with $p=2+\varepsilon$ to get existence.
\end{remark}

\begin{proof}
We focus on the complex-valued case; for the real-valued case, simply replace $\C^d$ by $\R^d$ below. We first prove that a solution exists for small times. Let $\rho\in(0,1]$ and $\varepsilon>0$ be constants to be specified later, and define
\[
B_{\rho,\varepsilon} = \{ \psi \in L^2([0,\rho],\C^d)\colon \|\psi\|_{L^2(0,\rho)} \le \varepsilon\}.
\]
Consider the map $F$ acting on elements $\psi\in B_{\rho,\varepsilon}$ by
\[
F(\psi) = g + K * p(\fdot,\psi).
\]
We write $\|\fdot\|_q = \|\fdot\|_{L^q(0,\rho)}$ for brevity in the following computations. The growth condition \eqref{eq:local Volterra:1} along with the Young, Cauchy--Schwarz, and triangle inequalities yield for $\psi,\widetilde\psi\in B_{\rho,\varepsilon}$
\begin{align}
\| F(\psi) \|_2 &\le \| g \|_2 + \| K \|_2\| p(\fdot,\psi) \|_1 \nonumber \\
&\le \| g \|_2 + \| K \|_2 \left( \|p(\fdot,0)\|_1 + \|\Pi_1\|_2\|\psi\|_2+\Theta_1\|\psi\|_2^2 \right) \label{eq:locVol:Y} \\
&\le \| g \|_2 + \| K \|_2 \left( \|p(\fdot,0)\|_1 + \|\Pi_1\|_{L^2(0,1)}\varepsilon + \Theta_1\varepsilon^2 \right) \nonumber
\end{align}
and
\begin{align*}
\| F(\psi) - F(\widetilde\psi) \|_2 &\le  \| K \|_2 \left( \|\Pi_1\|_2 + \Theta_1\left(\|\psi\|_2 + \|\widetilde\psi\|_2\right) \right) \|\psi - \widetilde\psi\|_2 \\
&\le \| K \|_2 \left( \|\Pi_1\|_{L^2(0,1)} + 2\Theta_1\varepsilon \right) \|\psi - \widetilde\psi\|_2.
\end{align*}
Choose $\varepsilon>0$ so that $1 +  \frac{\varepsilon}{2} +  \|\Pi_1\|_{L^2(0,1)}\varepsilon + \Theta_1\varepsilon^2<2$ and $\varepsilon\, ( \|\Pi_1\|_{L^2(0,1)} + 2\Theta_1\varepsilon ) < 2$. Then choose $\rho>0$ so that $\| g \|_2 \vee \| K \|_2 \vee \|p(\fdot,0)\|_1 \le \varepsilon/2$. This yields
\[
\| F(\psi) \|_2 \le \frac\varepsilon2 \left( 1 +  \frac{\varepsilon}{2} +\|\Pi_1\|_{L^2(0,1)}\varepsilon + \Theta_1\varepsilon^2\right) \le \varepsilon
\]
and
\[
\| F(\psi) - F(\widetilde\psi) \|_2 \le \kappa \|\psi - \widetilde\psi\|_2, \qquad \kappa = \frac\varepsilon2 \left( \|\Pi_1\|_{L^2(0,1)} + 2\Theta_1\varepsilon \right) < 1.
\]
Thus $F$ maps $B_{\rho,\varepsilon}$ to itself and is a contraction there, so Banach's fixed point theorem implies that $F$ has a unique fixed point $\psi\in B_{\rho,\varepsilon}$, which is a solution of \eqref{eq:VolIntEq}.

We now extend this to a unique non-continuable solution of \eqref{eq:VolIntEq}. Define the set
\[
J = \{T\in\R_+\colon \text{\eqref{eq:VolIntEq} has a solution $\psi\in L^2([0,T],\C^d)$ on $[0,T]$} \}.
\]
Then $0\in J$, and if $T\in J$ and $0\le S\le T$, then $S\in J$. Thus $J$ is a nonempty interval. Moreover, $J$ is open in $\R_+$. Indeed, pick $T\in J$, let $\psi$ be a solution on $[0,T]$, and set
\[
h(t) = g(T+t) + \int_0^T K(T+t-s)p(s,\psi(s))ds, \quad t\ge0,
\]
which lies in $L^2_{\rm loc}(\R_+,\C^d)$ by a calculation similar to \eqref{eq:locVol:Y}. By what we already proved, the equation
\[
\chi = h + K*p(\fdot+T,\chi)
\]
admits a solution $\chi\in L^2([0,\rho],\C^d)$ on $[0,\rho]$ for some $\rho>0$. Defining $\psi(t)=\chi(t-T)$ for $t\in(T,T+\rho]$, one verifies that $\psi$ solves \eqref{eq:VolIntEq} on $[0,T+\rho]$. Thus $T+\rho\in J$, so $J$ is open in $\R_+$ and hence of the form $J=[0,T_{\rm max})$ for some $0<T_{\rm max}\le\infty$ with $T_{\rm max}\notin J$. This yields a non-continuable solution $(\psi,T_{\rm max})$.

It remains to argue uniqueness. Pick $T\in\R_+$ and $\widetilde\psi\in L^2([0,T],\C^d)$ satisfying \eqref{eq:VolIntEq} on $[0,T]$. Then $T\in J$, so $T<T_{\rm max}$. Let $S$ be the supremum of all $S'\le T$ such that $\widetilde\psi=\psi$ on $[0,S']$. Then $\widetilde\psi=\psi$ on $[0,S]$ (almost everywhere, as elements of $L^2$). If $S<T$, then for $\rho>0$ sufficiently small we have $0<\|\psi-\widetilde\psi\|_{L^2(0,S+\rho)} \le \frac12 \|\psi-\widetilde\psi\|_{L^2(0,S+\rho)}$, a contradiction. Thus $S=T$, and uniqueness is proved.
\end{proof}

\begin{corollary}\label{C:linearVE}
Let $K\in L^{2}_{\rm loc}(\R_+,\C^{d\times d})$, $F\in L^2_{\rm loc}(\R_+,\C^d)$ and $G\in L^2_{\rm loc}(\R_+,\C^{d\times d})$. Suppose that  $p\colon \R_+\times \C^d\to\C^d$ is a Lipschitz continuous function in the second argument such that $p(\fdot,0)\in L^2_{\rm loc}(\R_+,\C^d)$. Then the equation  
\begin{equation*}
\chi =  F + K*\left(G p(\fdot, \chi) \right)
\end{equation*}
has a unique global solution $\chi\in L^{2}_{\rm loc}(\R_+,\C^{d})$. Moreover, if $K$ and $F$ are continuous on $[0,\infty)$ then $\chi$ is also continuous on $[0,\infty)$ and $\chi(0)=F(0)$.
\end{corollary}

\begin{proof}
Theorem~\ref{T:local Volterra} implies the existence and uniqueness of a non-continuable solution $(\chi,T_{\rm max})$. If $K$ and $F$ are continuous on $[0,\infty)$, then this solution is continuous on $[0,T_{\rm max})$ with $\chi(0)=F(0)$.  To prove that $T_{\rm max}=\infty$, observe that
\begin{equation}\label{E:ineqChi}
|\chi|\leq |F|+|K|*(|G|(|p(\fdot,0)|+\Theta|\chi|))
\end{equation}
for some positive constant $\Theta$.  Define the scalar non-convolution Volterra kernel $K'(t,s)=\Theta|K(t-s)||G(s)|\bm1_{s\le t}$. This is a Volterra kernel in the sense of \citet[Definition~9.2.1]{GLS:90} and
\begin{align}\label{E:youngNonconv}
\int_0^T \int_0^T \bm1_{s\le t}|K(t-s)|^2|G(s)|^2 ds\, dt \le \|\, K \|^2_{L^2(0,T)}\|\, G \|^2_{L^2(0,T)} 
\end{align}
for all $T >0$, by Young's inequality.
	Thus by \citet[Proposition~9.2.7(iii)]{GLS:90}, $K'$ is of type $L^2_{\rm loc}$, see \citet[Definition~9.2.2]{GLS:90}. In addition, it follows from  \citet[Corollary~9.3.16]{GLS:90} that $-K'$ admits a resolvent of type $L^2_{\rm loc}$ in the sense of \citet[Definition~9.3.1]{GLS:90}, which we denote by $R'$. Since $-K'$ is nonpositive,  it follows from \citet[Proposition~9.8.1]{GLS:90} that $R'$ is also nonpositive. The Gronwall type inequality in \citet[Lemma~9.8.2]{GLS:90} and~\eqref{E:ineqChi} then yield 
\begin{equation}\label{E:boundonChi}
|\chi(t)|\leq f'(t)-\int_0^t R'(t,s)f'(s)\,ds
\end{equation}
for $t\in[0,T_{\rm max}]$, where
\[
f'(t)=|F(t)|+\int_0^t |K(t-s)|\, |G(s)|\,|p(s,0)| ds.
\]
Since the function on the right-hand side of~\eqref{E:boundonChi} is in $L^2_{\rm loc}(\R_+,\R)$ due to \citet[Theorem~9.3.6]{GLS:90}, we conclude that $T_{\rm max}=\infty$.
\end{proof}

\section{Invariance results for Volterra integral equations}

\begin{lemma} \label{L:vol stability}
Fix $T<\infty$. Let $u\in\C^d$, $G\in L^2([0,T],\C^{d\times d})$, as well as $F^n\in L^2([0,T],\C^d)$ and $K^n\in L^2([0,T],\C^{d\times d})$ for $n=0,1,2,\ldots$. For each $n$, there exists a unique element $\chi^n\in L^2([0,T],\C^{d\times d})$ such that
\[
\chi^n = F^n + K^n*(G\chi^n).
\]
Moreover, if $F^n\to F^0$ and $K^n\to K^0$ in $L^2(0,T)$, then $\chi^n\to\chi^0$ in $L^2(0,T)$.
\end{lemma}

\begin{proof}
For any $K\in L^2([0,T],\C^{d\times d})$, define $K'(t,s)=K(t-s)G(s)\bm1_{s\le t}$. Arguing as in the proof of Corollary \ref{C:linearVE}, $K'$ is a Volterra  kernel of type  $L^2$ on $(0,T)$ since \eqref{E:youngNonconv}  still holds by  Young's inequality. In particular,
\begin{equation} \label{|||K|||}
	\iii K' \iii_{L^2(0,T)} \le \|\, K \|_{L^2(0,T)}\|\, G \|_{L^2(0,T)},
\end{equation}
where $\iii\fdot\iii_{L^2(0,T)}$ is defined in \citet[Definition~9.2.2]{GLS:90}. Invoking once again \citet[Corollary~9.3.16]{GLS:90}, $-K'$ admits a resolvent $R'$ of type $L^2$ on $(0,T)$. Due to \citet[Theorem~9.3.6]{GLS:90}, the unique solution in $L^2(0,T)$ of the equation
	\[
	\chi(t) = F(t) + \int_0^t K'(t,s)\chi(s)ds, \quad t\in[0,T],
	\]
	for a given $F\in L^2([0,T],\C^d)$, is
	\[
	\chi(t) = F(t) - \int_0^t R'(t,s)F(s)ds, \quad t\in[0,T].
	\]
	This proves the existence and uniqueness statement for the $\chi^n$. Next, assume $F^n\to F^0$ and $K^n\to K^0$ in $L^2(0,T)$.  Applying \eqref{|||K|||} with $K=K^n-K^0$ shows that $(K')^n \to (K')^0$ with respect to the norm $\iii\fdot \iii_{L^2(0,T)}$. An application of \citet[Corollary~9.3.12]{GLS:90} now shows that $\chi^n\to\chi^0$ in $L^2(0,T)$ as claimed.
\end{proof}

\begin{theorem} \label{T:positive_2}
Assume $K\in L^2_{\rm loc}(\R_+,\R^{d\times d})$ is diagonal with scalar kernels $K_i$ on the diagonal. Assume each $K_i$ satisfies \eqref{K_gamma} and the shifted kernels $\Delta_h K_i$ satisfy \eqref{eq:K orthant} for all $h\in[0,1]$. Let $u,v\in\R^d$, $F\in L^1_{\rm loc}(\R_+,\R^d)$ and $G\in L^2_{\rm loc}(\R_+,\R^{d\times d})$ be such that $u_i,v_i\ge0$, $F_i\ge0$, and $G_{ij}\ge0$ for all $i,j=1,\ldots,d$ and $i\ne j$. Then the linear Volterra equation 
\begin{equation}\label{E:chi}
\chi =  K u +v+ K*\left( F+ G\chi \right)
\end{equation}
has a unique solution $\chi\in L^2_{\rm loc}(\R_+,\R^d)$ with $\chi_i\ge0$ for $i=1,\ldots,d$.
\end{theorem}

\begin{proof}
Define kernels $K^n=K(\fdot+n^{-1})$ for $n\in\N$, which are diagonal with scalar kernels on the diagonal that satisfy \eqref{eq:K orthant}. Example~\ref{ex:kernels}\ref{ex:kernels:vi} shows that the scalar kernels on the diagonal of $K^n$ also satisfy \eqref{K_gamma}. Lemma~\ref{L:vol stability} shows that \eqref{E:chi} (respectively \eqref{E:chi} with $K$ replaced by $K^n$) has a unique solution $\chi$ (respectively $\chi^n$), and that $\chi^n\to\chi$ in $L^2(\R_+,\R^d)$. Therefore, we can suppose without loss of generality that $K$ is continuous on $[0,\infty)$ with $K_i(0)\ge 0$. To shows that $\chi$ takes values in $\R^d_+$, it is therefore enough to consider the case where $K$ is continuous on $[0,\infty)$ with $K_i(0)\ge 0$ for all $i$.

For $x\in \R^d$ define $b(x)=F+Gx$. For all positive $n$, Corollary~\ref{C:linearVE} implies that there exists a unique solution $\chi^n\in L^2_{\rm loc}(\R_+,\R^d)$ of the equation
\[
\chi^n= Ku +v+ K* b((\chi^n-n^{-1})^+),
\]
and that $\chi^n$ is continuous on $[0,\infty)$ with $\chi^n_i(0)=K_i(0)u_i+v_i\geq 0$ for $i=1,\ldots,d$. We claim that $\chi^n$ is $\R_+^d$ valued for all $n$. Indeed,  arguing as in the proof of Theorem~\ref{T:existence orthant}, we can show that if $L_i$ denotes the resolvent of the first kind of $K_i$, then $(\Delta_h K_i * L_i)(t)$ is right-continuous, nonnegative, bounded by $1$, and nondecreasing in $t$ for any $h\ge0$. Fix $n$ and define $Z=\int b((\chi^n-n^{-1})^+)\,dt$. The argument of Lemma~\ref{L:ZX} shows that for all $h\geq 0$ and $i=1,\ldots,d$,
\begin{equation}\label{E:IBP_Inv_Ricatti}
\begin{aligned}
\Delta_h K_i *dZ_i&=(\Delta_h K_i * L_i)(0)K_i*dZ_i+d(\Delta_h K_i * L_i)*K_i*dZ_i\\
&=(\Delta_h K_i * L_i)(0)\chi^n_i + d(\Delta_h K_i * L_i)*\chi^n_i \\
&\quad -u_i\left((\Delta_h K_i * L_i)(0)K_i + d(\Delta_h K_i * L_i)*K_i \right))-v_i\Delta_h K_i * L_i.
\end{aligned}
\end{equation}
Convolving the quantity $d(\Delta_h K_i * L_i)*K_i$ first by $L_i$, then by $K_i$, and comparing densities of the resulting absolutely continuous functions, we deduce that
\[
d(\Delta_h K_i * L_i)*K_i=\Delta_h K_i-(\Delta_h K_i * L_i)(0)K_i \quad \text{a.e.}
\]
Plugging this identity into~\eqref{E:IBP_Inv_Ricatti} yields
\begin{equation}\label{E:IBP_Inv_Ricatti_2}
\Delta_h K_i *dZ_i=(\Delta_h K_i * L_i)(0)\chi^n_i+d(\Delta_h K_i * L_i)*\chi^n_i-u_i\Delta_h K_i-v_i\Delta_h K_i * L_i.
\end{equation}
Define $\tau=\inf\{t\ge0\colon \chi^n_t\notin\R^d_+\}$ and assume for contradiction that $\tau<\infty$. Then
\begin{equation}\label{E:shiifted_RIcatti_sigma}
\begin{split}
\chi^n(\tau+h)&=\Delta_h K(\tau)u+v + (K*dZ)_{\tau+h} \\
&= \Delta_h K(\tau) u+v+ (\Delta_h K * dZ)_\tau + \int_0^h K(h-s) d Z_{\tau+s}
\end{split}
\end{equation}
for any $h\ge0$. By definition of $\tau$, the identities~\eqref{E:IBP_Inv_Ricatti_2} and~\eqref{E:shiifted_RIcatti_sigma} imply
\[
\chi^n_i(\tau+h)\geq \int_0^h K_i(h-s)b_i((\chi^n(\tau+s)-n^{-1})^+)\,ds,\quad i=1,\ldots,d.
\]
As in the proof of Theorem~\ref{T:existence orthant}, these inequalities lead to a contradiction. Hence $\tau=\infty$ and $\chi^n$ is $\R_+^d$-valued for all $n$. 

To conclude that $\chi$ is $\R^d_+$-valued it suffices to prove that $\chi^n$ converges to $\chi$ in $L^2([0,T],\R^d)$ for all $T\in\R_+$. To this end we write
\[
\chi-\chi^n=K*\left(G(\chi^n-(\chi^n-n^{-1})^+)+G(\chi-\chi^n)\right),
\]
from which we infer
\[
|\chi-\chi^n|\leq \frac{\sqrt{d}}{n}|K|*|G|+|K|*(|G||\chi-\chi^n|).
\]
The same argument as in the proof of Corollary~\ref{C:linearVE} shows that
\begin{equation}\label{E:boundChiChin}
|\chi-\chi^n|\leq \frac{\sqrt{d}}{n}\left(F'-\int_0^{\fdot} R'(\fdot,s)F'(s)\,ds\right),
\end{equation}
where $R'$ is the nonpositive resolvent of type $L^2_{\rm loc}$ of $K'(t,s)=|K(t-s)||G(s)|\bm1_{s\le t}$, and $F'=|K|*|G|$. Since the right-hand side of~\eqref{E:boundChiChin} is in $L^2_{\rm loc}(\R_+,\R)$ in view of \citet[Theorem~9.3.6]{GLS:90}, we conclude that $\chi^n$ converges to $\chi$ in $L^2([0,T],\R^d)$ for all $T\in\R_+$.
\end{proof}

% ----------------------------------------------------------------
\bibliographystyle{plainnat}
\bibliography{bibl}

\begin{thebibliography}{34}
\providecommand{\natexlab}[1]{#1}
\providecommand{\url}[1]{\texttt{#1}}
\expandafter\ifx\csname urlstyle\endcsname\relax
  \providecommand{\doi}[1]{doi: #1}\else
  \providecommand{\doi}{doi: \begingroup \urlstyle{rm}\Url}\fi

\bibitem[Abi~Jaber et~al.(2018)Abi~Jaber, Bouchard, and Illand]{AJBI:16}
Eduardo Abi~Jaber, Bruno Bouchard, and Camille Illand.
\newblock Stochastic invariance of closed sets with non-{L}ipschitz
  coefficients.
\newblock \emph{Stochastic Processes and their Applications}, 2018.
\newblock \doi{10.1016/j.spa.2018.06.003}.
\newblock URL
  \url{http://www.sciencedirect.com/science/article/pii/S0304414918302758}.

\bibitem[Bayer et~al.(2016)Bayer, Friz, and Gatheral]{Bayeretal2016}
Christian Bayer, Peter Friz, and Jim Gatheral.
\newblock Pricing under rough volatility.
\newblock \emph{Quantitative Finance}, 16\penalty0 (6):\penalty0 887--904,
  2016.
\newblock \doi{10.1080/14697688.2015.1099717}.
\newblock URL \url{http://dx.doi.org/10.1080/14697688.2015.1099717}.

\bibitem[Bennedsen et~al.(2016)Bennedsen, Lunde, and Pakkanen]{BLP:16}
Mikkel Bennedsen, Asger Lunde, and Mikko~S. Pakkanen.
\newblock Decoupling the short- and long-term behavior of stochastic
  volatility.
\newblock \emph{arXiv preprint arXiv:1610.00332}, 2016.

\bibitem[Berger and Mizel(1980{\natexlab{a}})]{BM:80:1}
Marc~A. Berger and Victor~J. Mizel.
\newblock Volterra equations with {I}t\^o integrals. {I}.
\newblock \emph{J. Integral Equations}, 2\penalty0 (3):\penalty0 187--245,
  1980{\natexlab{a}}.
\newblock ISSN 0163-5549.
\newblock URL \url{http://www.jstor.org/stable/26164035}.

\bibitem[Berger and Mizel(1980{\natexlab{b}})]{BM:80:2}
Marc~A. Berger and Victor~J. Mizel.
\newblock Volterra equations with {I}t\^o integrals. {II}.
\newblock \emph{J. Integral Equations}, 2\penalty0 (4):\penalty0 319--337,
  1980{\natexlab{b}}.
\newblock ISSN 0163-5549.
\newblock URL \url{http://www.jstor.org/stable/26164044}.

\bibitem[Billingsley(1999)]{B:99}
Patrick Billingsley.
\newblock \emph{Convergence of probability measures}.
\newblock Wiley Series in Probability and Statistics: Probability and
  Statistics. John Wiley \& Sons, Inc., New York, second edition, 1999.
\newblock ISBN 0-471-19745-9.
\newblock \doi{10.1002/9780470316962}.
\newblock URL \url{http://dx.doi.org/10.1002/9780470316962}.
\newblock A Wiley-Interscience Publication.

\bibitem[Comte et~al.(2012)Comte, Coutin, and Renault]{com_cou_ren_12}
F.~Comte, L.~Coutin, and E.~Renault.
\newblock Affine fractional stochastic volatility models.
\newblock \emph{Ann. Finance}, 8\penalty0 (2-3):\penalty0 337--378, 2012.
\newblock ISSN 1614-2446.
\newblock \doi{10.1007/s10436-010-0165-3}.
\newblock URL \url{https://doi.org/10.1007/s10436-010-0165-3}.

\bibitem[Coutin and Decreusefond(2001)]{cou_dec_01}
Laure Coutin and Laurent Decreusefond.
\newblock Stochastic {V}olterra equations with singular kernels.
\newblock In \emph{Stochastic analysis and mathematical physics}, volume~50 of
  \emph{Progr. Probab.}, pages 39--50. Birkh\"{a}user Boston, Boston, MA, 2001.

\bibitem[Cuchiero et~al.(2011)Cuchiero, Filipovi\'c, Mayerhofer, and
  Teichmann]{CFMT:11}
Christa Cuchiero, Damir Filipovi\'c, Eberhard Mayerhofer, and Josef Teichmann.
\newblock Affine processes on positive semidefinite matrices.
\newblock \emph{Ann. Appl. Probab.}, 21\penalty0 (2):\penalty0 397--463, 2011.
\newblock ISSN 1050-5164.
\newblock \doi{10.1214/10-AAP710}.
\newblock URL \url{http://dx.doi.org/10.1214/10-AAP710}.

\bibitem[Decreusefond(2002)]{D:02}
Laurent Decreusefond.
\newblock Regularity properties of some stochastic {V}olterra integrals with
  singular kernel.
\newblock \emph{Potential Analysis}, 16\penalty0 (2):\penalty0 139--149, 2002.
\newblock \doi{10.1023/A:1012628013041}.
\newblock URL \url{https://doi.org/10.1023/A:1012628013041}.

\bibitem[Duffie et~al.(2003)Duffie, Filipovi\'c, and Schachermayer]{DFS:03}
Darrell Duffie, Damir Filipovi\'c, and Walter Schachermayer.
\newblock Affine processes and applications in finance.
\newblock \emph{Ann. Appl. Probab.}, 13\penalty0 (3):\penalty0 984--1053, 2003.
\newblock ISSN 1050-5164.
\newblock \doi{10.1214/aoap/1060202833}.
\newblock URL \url{http://dx.doi.org/10.1214/aoap/1060202833}.

\bibitem[El~Euch and Rosenbaum(2016)]{EER:06}
Omar El~Euch and Mathieu Rosenbaum.
\newblock The characteristic function of rough {H}eston models.
\newblock \emph{Mathematical Finance}, 2016.
\newblock \doi{10.1111/mafi.12173}.
\newblock URL \url{https://doi.org/10.1111/mafi.12173}.

\bibitem[El~Euch and Rosenbaum(2018)]{EER:07}
Omar El~Euch and Mathieu Rosenbaum.
\newblock Perfect hedging in rough {H}eston models.
\newblock \emph{The Annals of Applied Probability}, 28\penalty0 (6):\penalty0
  3813--3856, 2018.
\newblock \doi{10.1214/18-AAP1408}.
\newblock URL \url{https://doi.org/10.1214/18-AAP1408}.

\bibitem[{E}l Euch et~al.(2018){E}l Euch, Fukasawa, and
  Rosenbaum]{EluchFukasawaRosenbaum2016}
Omar {E}l Euch, Masaaki Fukasawa, and Mathieu Rosenbaum.
\newblock The microstructural foundations of leverage effect and rough
  volatility.
\newblock \emph{Finance and Stochastics}, 22\penalty0 (2):\penalty0 241--280,
  2018.
\newblock \doi{10.1007/s00780-018-0360-z}.
\newblock URL \url{https://doi.org/10.1007/s00780-018-0360-z}.

\bibitem[Filipovi\'c(2009)]{F:09}
Damir Filipovi\'c.
\newblock \emph{Term-structure models}.
\newblock Springer Finance. Springer-Verlag, Berlin, 2009.
\newblock ISBN 978-3-540-09726-6.
\newblock \doi{10.1007/978-3-540-68015-4}.
\newblock URL \url{http://dx.doi.org/10.1007/978-3-540-68015-4}.
\newblock A graduate course.

\bibitem[Gatheral et~al.(2018)Gatheral, Jaisson, and
  Rosenbaum]{volatilityrough2014}
Jim Gatheral, Thibault Jaisson, and Mathieu Rosenbaum.
\newblock Volatility is rough.
\newblock \emph{Quantitative Finance}, 18\penalty0 (6):\penalty0 933--949,
  2018.
\newblock \doi{10.1080/14697688.2017.1393551}.
\newblock URL \url{https://doi.org/10.1080/14697688.2017.1393551}.

\bibitem[Gripenberg et~al.(1990)Gripenberg, Londen, and Staffans]{GLS:90}
Gustaf Gripenberg, Stig-Olof Londen, and Olof Staffans.
\newblock \emph{Volterra integral and functional equations}, volume~34 of
  \emph{Encyclopedia of Mathematics and its Applications}.
\newblock Cambridge University Press, Cambridge, 1990.
\newblock ISBN 0-521-37289-5.
\newblock \doi{10.1017/CBO9780511662805}.
\newblock URL \url{http://dx.doi.org/10.1017/CBO9780511662805}.

\bibitem[Guennoun et~al.(2018)Guennoun, Jacquier, Roome, and Shi]{GJR:17}
Hamza Guennoun, Antoine Jacquier, Patrick Roome, and Fangwei Shi.
\newblock Asymptotic behavior of the fractional {H}eston model.
\newblock \emph{SIAM Journal on Financial Mathematics}, 9\penalty0
  (3):\penalty0 1017--1045, 2018.
\newblock \doi{10.1137/17M1142892}.
\newblock URL \url{https://doi.org/10.1137/17M1142892}.

\bibitem[Heston(1993)]{heston1993closed}
Steven~L. Heston.
\newblock A closed-form solution for options with stochastic volatility with
  applications to bond and currency options.
\newblock \emph{The Review of Financial Studies}, 6\penalty0 (2):\penalty0
  327--343, 1993.
\newblock ISSN 08939454, 14657368.
\newblock URL \url{http://www.jstor.org/stable/2962057}.

\bibitem[Hofmanov\'a and Seidler(2012)]{HS:12}
Martina Hofmanov\'a and Jan Seidler.
\newblock On weak solutions of stochastic differential equations.
\newblock \emph{Stoch. Anal. Appl.}, 30\penalty0 (1):\penalty0 100--121, 2012.
\newblock ISSN 0736-2994.
\newblock \doi{10.1080/07362994.2012.628916}.
\newblock URL \url{http://dx.doi.org/10.1080/07362994.2012.628916}.

\bibitem[Larsson and Kr\"uhner(2018)]{KL:17}
Martin Larsson and Paul Kr\"uhner.
\newblock Affine processes with compact state space.
\newblock \emph{Electronic Journal of Probability}, 23, 2018.
\newblock \doi{10.1214/18-EJP156}.
\newblock URL \url{https://doi.org/10.1214/18-EJP156}.

\bibitem[Marinelli et~al.(2010)Marinelli, Pr{\'e}v{\^o}t, and
  R{\"o}ckner]{MARINELLI2010616}
Carlo Marinelli, Claudia Pr{\'e}v{\^o}t, and Michael R{\"o}ckner.
\newblock Regular dependence on initial data for stochastic evolution equations
  with multiplicative {P}oisson noise.
\newblock \emph{Journal of Functional Analysis}, 258\penalty0 (2):\penalty0 616
  -- 649, 2010.
\newblock ISSN 0022-1236.
\newblock \doi{https://doi.org/10.1016/j.jfa.2009.04.015}.
\newblock URL
  \url{http://www.sciencedirect.com/science/article/pii/S0022123609001943}.

\bibitem[Mytnik and Neuman(2011)]{MN:11}
Leonid Mytnik and Eyal Neuman.
\newblock Sample path properties of {V}olterra processes.
\newblock \emph{arXiv preprint arXiv:1101.4969}, 2011.

\bibitem[Mytnik and Salisbury(2015)]{MS:15}
Leonid Mytnik and Thomas~S. Salisbury.
\newblock Uniqueness for {V}olterra-type stochastic integral equations.
\newblock \emph{arXiv preprint arXiv:1502.05513}, 2015.

\bibitem[Pardoux and Protter(1990)]{PP:90}
\'Etienne Pardoux and Philip Protter.
\newblock Stochastic {V}olterra equations with anticipating coefficients.
\newblock \emph{Ann. Probab.}, 18\penalty0 (4):\penalty0 1635--1655, 1990.
\newblock ISSN 0091-1798.
\newblock URL
  \url{http://links.jstor.org/sici?sici=0091-1798(199010)18:4<1635:SVEWAC>2.0.CO;2-9&origin=MSN}.

\bibitem[Peszat and Zabczyk(2007)]{PZ07}
S.~Peszat and J.~Zabczyk.
\newblock \emph{Stochastic Partial Differential Equations with {L}{\'e}vy
  Noise: An Evolution Equation Approach}.
\newblock Encyclopedia of Mathematics and its Applications. Cambridge
  University Press, 2007.
\newblock \doi{10.1017/CBO9780511721373}.

\bibitem[Protter(1985)]{P:85}
Philip Protter.
\newblock Volterra equations driven by semimartingales.
\newblock \emph{Ann. Probab.}, 13\penalty0 (2):\penalty0 519--530, 1985.
\newblock ISSN 0091-1798.
\newblock URL
  \url{http://links.jstor.org/sici?sici=0091-1798(198505)13:2<519:VEDBS>2.0.CO;2-3&origin=MSN}.

\bibitem[Protter(2004)]{pro_04}
Philip~E. Protter.
\newblock \emph{Stochastic integration and differential equations}, volume~21
  of \emph{Applications of Mathematics (New York)}.
\newblock Springer-Verlag, Berlin, second edition, 2004.
\newblock ISBN 3-540-00313-4.
\newblock Stochastic Modelling and Applied Probability.

\bibitem[Revuz and Yor(1999)]{RY:99}
Daniel Revuz and Marc Yor.
\newblock \emph{Continuous martingales and {B}rownian motion}, volume 293 of
  \emph{Grundlehren der Mathematischen Wissenschaften [Fundamental Principles
  of Mathematical Sciences]}.
\newblock Springer-Verlag, Berlin, third edition, 1999.
\newblock ISBN 3-540-64325-7.
\newblock \doi{10.1007/978-3-662-06400-9}.
\newblock URL \url{http://dx.doi.org/10.1007/978-3-662-06400-9}.

\bibitem[Spreij and Veerman(2010)]{SV:10}
Peter Spreij and Enno Veerman.
\newblock The affine transform formula for affine jump-diffusions with a
  general closed convex state space.
\newblock \emph{arXiv preprint arXiv:1005.1099}, 2010.

\bibitem[Spreij and Veerman(2012)]{SV:12}
Peter Spreij and Enno Veerman.
\newblock Affine diffusions with non-canonical state space.
\newblock \emph{Stoch. Anal. Appl.}, 30\penalty0 (4):\penalty0 605--641, 2012.
\newblock ISSN 0736-2994.
\newblock \doi{10.1080/07362994.2012.684322}.
\newblock URL \url{http://dx.doi.org/10.1080/07362994.2012.684322}.

\bibitem[Veraar(2012)]{V:12}
Mark Veraar.
\newblock The stochastic {F}ubini theorem revisited.
\newblock \emph{Stochastics}, 84\penalty0 (4):\penalty0 543--551, 2012.
\newblock ISSN 1744-2508.
\newblock \doi{10.1080/17442508.2011.618883}.
\newblock URL \url{http://dx.doi.org/10.1080/17442508.2011.618883}.

\bibitem[Wang(2008)]{W:08}
Zhidong Wang.
\newblock Existence and uniqueness of solutions to stochastic {V}olterra
  equations with singular kernels and non-{L}ipschitz coefficients.
\newblock \emph{Statistics \& Probability Letters}, 78\penalty0 (9):\penalty0
  1062--1071, 2008.
\newblock \doi{https://doi.org/10.1016/j.spl.2007.10.007}.
\newblock URL
  \url{http://www.sciencedirect.com/science/article/pii/S0167715207003665}.

\bibitem[Zhang(2010)]{Z:10}
Xicheng Zhang.
\newblock Stochastic {V}olterra equations in {B}anach spaces and stochastic
  partial differential equation.
\newblock \emph{J. Funct. Anal.}, 258\penalty0 (4):\penalty0 1361--1425, 2010.
\newblock ISSN 0022-1236.
\newblock \doi{10.1016/j.jfa.2009.11.006}.
\newblock URL \url{http://dx.doi.org/10.1016/j.jfa.2009.11.006}.

\end{thebibliography}
\end{document}